\index{}

\documentclass[12pt,leqno]{amsart}
\usepackage{amssymb , amsmath, amsthm}
\usepackage{amsmath}
\usepackage{amsbsy}
\usepackage{amssymb}

\usepackage[active]{srcltx}

\usepackage{amsmath}
\usepackage{amsfonts}
\usepackage{amssymb}
\usepackage{amstext}
\usepackage{amsbsy}
\usepackage{epsf}

\usepackage{xcolor}
\usepackage[colorlinks=true,linkcolor=blue,citecolor=red,filecolor=yellow,
linkcolor=purple,urlcolor=green]{hyperref}

\newtheorem{theorem}{Theorem}[section]
\newtheorem{proposition}[theorem]{Proposition}

\theoremstyle{definition}
\newtheorem{definition}[theorem]{Definition}
\newtheorem{example}[theorem]{Example}
\newtheorem{remark}[theorem]{Remark}

\def\r{\mathbb R}

\def\h{\mathbb H}

\def\Hip{\mathbb H}

\let\hip=\Hip

\newcommand{\R}{\mathbb R}
\newcommand{\M}{\mathbb M}
\newcommand{\E}{\mathbb E}

\newcommand{\di}{\mathbb D}
\newcommand{\N}{\mathbb N}

\newcommand{\wt}{\widetilde}
\newcommand{\wh}{\widehat}

\renewcommand{\Im}{{\rm Im}}

\let\h=\hip

\def\rmd{\mathop{\rm d\kern -1pt}\nolimits}
\def\rme{\mathop{\rm e\kern -1pt}\nolimits}

\DeclareMathOperator{\dist}{dist}

\DeclareMathOperator{\diam}{diam}

\def\bel{ \medskip
 \centerline{$ \ast \hbox to 1.0cm{}\ast \hbox to 1.0cm{}\ast $}
}

\def\longerrightarrow{-\kern-5pt\longrightarrow}

\def\star{\lower 1pt\hbox{*}}
\def \nulset {
\raise 1pt\hbox{ \hskip -3pt$\not$\kern -0.2pt \raise
.7pt\hbox{${\scriptstyle\bigcirc}$}}}

\newcommand{\hd}{\mathbb{H}^2}

\newcommand{\hi}[1]{\mathbb{H}^#1}

\newcommand{\ov}[1]{\overline{#1}}

\let\leq=\leqslant
\let\geq=\geqslant

\begin{document}
\title[reflection principle]{A reflection principle for minimal surfaces in   
smooth three manifolds }

\author[R. Sa
Earp $\ $ and $\ $ E. Toubiana  ]{
 Ricardo Sa Earp and
Eric Toubiana}

 \address{Departamento de Matem\'atica \newline
  Pontif\'\i cia Universidade Cat\'olica do Rio de Janeiro\newline
Rio de Janeiro \newline
22451-900 RJ \newline
 Brazil }
\email{rsaearp@gmail.com}

\address{Institut de Math\'ematiques de Jussieu - Paris Rive Gauche \newline
Universit\'e Paris Diderot - Paris 7 \newline
Equipe G\'eom\'etrie et Dynamique,  UMR 7586 \newline
B\^atiment Sophie Germain \newline
Case 7012 \newline
75205 Paris Cedex 13 \newline
France}
\email{eric.toubiana@imj-prg.fr}

\thanks{
\\
The  first and second authors were partially supported by CNPq
 of Brasil.
}

\date{\today}

\begin{abstract}
 We prove a reflection principle for minimal surfaces in smooth 
 (non necessarily analytic) three manifolds and we give an explicit  application when the ambient space is just a smooth  manifold. 
\end{abstract}


\maketitle

\section{Introduction}

In this paper we prove a reflection principle for minimal surfaces in a smooth (non 
necessarily analytic) setting. The precise statement is given in the Theorem 1.1  below. 
One of the motivations to seek such a reflection principle is the fact that there are many 
three dimensional smooth (non analytic) Riemannian manifolds  endowed with natural  
reflections about  a geodesic line (see, exemple 2.3-(3)).  We  derive an application  
when the ambient is just a smooth manifold, using  our main theorem combined with some 
classical theorems about area minimizing solution of the Plateau problem, see Proposition 
3.4.

\smallskip 

A form of the reflection principle says that if a minimal surface in 
$\R^3$ contains a segment 
of straight line $L$, then the minimal surface is invariant by reflection with respect to 
 L, see \cite{H-K}. 
Such a  form was generalized by Leung \cite[Thm. 1]{Le} for analytic 
manifolds.

\smallskip 

We are interested in  the knowledge of a 
reflection principle across  a geodesic 
line  in the boundary of a minimal surface. We assume the surface is contained 
in a smooth non necessarily analytic manifold. We assume also that the geodesic 
line admits a reflection in the ambient space (Definition \ref{D.reflection}).

\smallskip

First, we emphasize that the   ``reflection 
principle'' in the 
Euclidean space with only the usual hypothesis that the minimal surface contains 
a segment $L$ of a straight 
line in its (topological) boundary, is  not established. In fact, 
 even 
with the 
strongest assumption that the surface is an embedded disk up to the boundary 
segment $L$, as far as we know, there is no proof 
of the "reflection principle" in the Euclidean space.

\smallskip

Of course, when we impose some additional conditions a proof can be done. For 
example, when 
we know that the minimal immersion is conformal in the interior and continuous 
up to the boundary segment $L$ or when 
the surface with its boundary $L$ is a minimal graph, continuous up to  $L$. 

\smallskip

In fact,  the reflection principle for 
conformal  minimal immersions  in $\R^3$ is 
a generalization of the well-known Schwarz reflection principle for harmonic 
functions. 
The proof 
uses the fact that the coordinates of a 
conformal  minimal immersion  in Euclidean space are harmonic, then the Schwarz 
principle 
 for harmonic functions is applied. See an elegant 
deduction in \cite[Thm. 1, Sec. 4.8]{D-H-S} or in \cite[Lemma 7.3]{Osserman}.

\smallskip 

When the ambient space is the sphere $\mathbb S^3$, Lawson produced a proof  
following 
the same idea of the proof in the Euclidean case (which, in fact, 
holds in the three-dimensional 
hyperbolic space). The precise statement is as follows: When the conformal 
minimal  
immersion in the sphere $\mathbb S^3$ contains an arc of geodesic $L$ of the 
ambient space on its boundary and has $C^2$ regularity up to 
this arc, then the surface can be extended analytically 
by  
reflection in $L$ \cite[Prop. 3.1]{La}. The proof makes use of  a Lichtenstein 
theorem, see \cite[Thm. 1]{Hi-Mo}.

\smallskip 

The reflection principle  has been used by several authors (including  
the present authors) 
in the theory of minimal surfaces 
in 
homogeneous three-dimensional spaces, see for example Rosenberg \cite{Ros}, 
Abresh-Rosenberg \cite{A-R}, Menezes \cite{Men}, 
Mazet-Rodr\'\i guez-Rosenberg \cite{Ma-Rod-Ros}.

\smallskip 

On the other hand, the authors have established the  reflection principle  for 
 minimal vertical 
graphs, when the ambient space is the product space $\Hip^n \times  \R$,  where 
$\Hip^n $ is the
$n$-dimensional hyperbolic space,
see \cite[Lemma 3.6]{SE-T6}. The proof also works in $\R^n \times \R$.

Furthermore the authors use the reflection principle to 
construct Scherk type minimal hypersurface in $\Hip^n \times  \R$
\cite[Theorem 5.10]{SE-T6}.


\bigskip 

\bigskip

In the  statement of our Main Theorem below we use the notion of a reflection 
$I_\gamma$ about a 
geodesic $\gamma$ in a $C^\infty$ Riemannian manifold $(M,g)$. 
We denote by $U_\gamma \subset M$ 
the domain of definition  of $I_\gamma$, 
see Definition 
\ref{D.reflection}.

\begin{theorem}[Main Theorem]\label{T.Main.Theorem}
Let $(M,g)$ be a $C^\infty$ Riemannian three manifold. 
Let $\gamma\subset M$ be an open geodesic arc which admits a 
reflection 
$I_\gamma$. 

Let $S\subset U_\gamma$ be an embedded minimal surface. We assume that  
  $S\cup \gamma$ is a $C^1$ surface with 
boundary.

\smallskip

Then the reflection of $S$ about $\gamma$ gives rise to a $C^\infty$ 
continuation of 
$S$ across $\gamma$. That is, $S\cup \gamma \cup I_\gamma(S)$ is a smooth  
immersed minimal surface 
which is embedded near $\gamma$. 
\end{theorem}

\medskip

\begin{theorem}[Main Theorem bis]\label{T.Main.Theorem bis}
 Let $M$ be a $C^\infty$ Riemannian three manifold and 
let $\gamma\subset M$ be an open geodesic arc which admits a 
reflection $I_\gamma$. 

Let $S\subset U_\gamma$ be an embedded minimal surface such that 
$S\cup \gamma$ is a $C^0$ surface with 
boundary.
We assume that $S\cup \gamma$ is the 
graph of a $C^0$ function  $x_3 =f(x_1,x_2)$ for some 
local coordinates $(x_1,x_2,x_3)$ of $M$. 

We assume also that 
$f$ restricted to the projection of $S$ is a $C^2$ function with bounded 
gradient.

Then the reflection of $S$ about $\gamma$ gives rise to a $C^\infty$ 
continuation of 
$S$ across $\gamma$. That is, $S\cup \gamma \cup I_\gamma(S)$ is a smooth  
immersed minimal surface 
which is a graph near $\gamma$.
\end{theorem}

\bigskip 

We point out that a crucial tool in the proof of the above theorems is a 
 H\"older gradient regularity up to the boundary for solutions of 
 Dirichlet problems for quasilinear elliptic equations, see Theorems 
\ref{T.Holder-Trudinger}, \ref{T.Schauder} and Remark \ref{R.local}.

\subsection*{Acknowledgements}
{\Small   The second author wishes
to thank the {\em Departamento de Matem\'atica da
PUC-Rio} for their kind hospitality. }

\section{Analytic and geometric background} 

\subsection{Geodesic arc of reflection}
$   $

Throughout this paper, $M$ is  a connected $C^\infty$ manifold of dimension 
three, 
equipped with a $C^\infty$ metric $g$.

\bigskip 

\begin{definition}\label{D.reflection}
Let $(M,g)$ be a complete $C^\infty$ Riemannian three manifold.   

We say that an open geodesic arc  $\gamma\subset M$ admits a 
{\em reflection} if there exist an oriented open subset $U_\gamma\subset M$ 
containing 
$\gamma$, and 
a non trivial isometry $I_\gamma : (U_\gamma,g)\longrightarrow (U_\gamma,g)$ 
such that 
\begin{itemize}
\item $I_\gamma$ is orientation preserving,

 \item $I_\gamma (p) =p$ for any $p\in \gamma$,
 
\item $ I_\gamma \circ I_\gamma= {\rm Id}_{U_\gamma}$.
\end{itemize}

\end{definition}
 
Next we describe some relevant examples. 

\begin{example}
\begin{enumerate}
 \item We denote by $\hd$ the hyperbolic plane and $g_\h$ 
 its hyperbolic metric.

 Let $M= \hd \times \R$  endowed  with the product metric 
 $g:=g_\h + dt^2$.
 
 The natural reflections about vertical or horizontal geodesic lines 
in $\hd\times \R$ satisfy 
the conditions of Definition \ref{D.reflection}. We observe that there is no 
other geodesic 
lines of reflection.

\smallskip 

\item For any $\kappa \leq 0$ we denote by $\M^2(\kappa)$ the complete and 
simply connected surface with constant intrinsic curvature $\kappa$.

We set  $M=\E (\kappa, \tau)$, $\kappa \leq 0$, $\tau >0$, the simply 
connected and complete homogeneous three manifold, see for example 
\cite{Daniel}. More precisely, $\E (\kappa, \tau)$ has a four dimensional 
isometry group, it is a fibration over  $\M^2(\kappa)$, 
the canonical  projection 
$\E (\kappa, \tau) \longrightarrow \M^2(\kappa)$ 
is a Riemannian submersion and the bundle curvature is 
$\tau$.

We treat separately the cases $\kappa <0$ and $\kappa=0$.
\begin{enumerate}
 \item We suppose $\kappa< 0$, (in particular for 
 $\kappa=-1$ we have $\E (-1, \tau)=\wt{\rm PSL}_2(\R,\tau)$).

 We choose the disk model for $\M^2(\kappa)$, that is the open disk of radius 
 $2/\sqrt{-\kappa}$. We denote  : $\M^2(\kappa):=\di (2/\sqrt{-\kappa})$  
provided  with the metric 
\begin{equation*}
 ds^2=\lambda^2 (x,y) (dx^2 +dy^2),\ \ \mathrm{where} \ \ 
\lambda (x,y)=\frac{1}{1+\kappa \frac{x^2+y^2}{4}}.
\end{equation*}
Then we have  $\E (\kappa, \tau)=\di (2/\sqrt{-\kappa})\times \R$ provided with 
the metric 
\begin{align*}
 g &=    \lambda^2 (dx^2+dy^2) +\Bigl(- \frac{2\tau}{\kappa}
\frac{\lambda_y}{\lambda} dx +\frac{2\tau}{\kappa}
\frac{\lambda_x}{\lambda} dy +dt\Bigr)^2 \\
&= \lambda^2 (dx^2+dy^2) + \Bigl( \lambda \tau 
\big(ydx -xdy \big) + dt\Bigr)^2.
\end{align*}
The isometries of $(\E (\kappa, \tau), g)$ are given by
(setting $z=x+iy$)
\begin{equation}
F(z,t)=(f(z), t+ \frac{2\tau}{\kappa} \arg \frac{\partial f}{\partial z} (z) +c)
\end{equation}  
where $f$ is a positive isometry of  $(\di (2/\sqrt{-\kappa}), ds^2)$, 
and 
\begin{equation}
G(z,t)=( g(z), -t - \frac{2\tau}{\kappa} \arg \frac{\partial \, 
\bar{g}}{\partial z} (z) +c),
\end{equation} 
where $g$ is a negative isometry of  $(\di (2/\sqrt{-\kappa}), ds^2)$,
and $c \in \R$. Observe that any isometry is orientation preserving. 

In particular $F(z,t)=(-z,t)$ is an isometry. Set $L_0:=\{(0,t), \ t\in \R\}$. 
Observe that each point of $L_0$ is a fixed point for $F$. Therefore $L_0$ is a 
geodesic line. Since $F\circ F={\rm Id}$, $F$ is a reflection about $L_0$.  

Now let $z_0\in \di (2/\sqrt{-\kappa})$ be any point. Let $h$ be a positive 
isometry of $\di (2/\sqrt{-\kappa})$ such that $h(z_0)=0$. We set 
$H(z,t)=(h(z),t+ \frac{2\tau}{\kappa} \arg \frac{\partial h}{\partial z} 
(z))$ and 
$L_{z_0}:=\{(z_0,t), \ t\in \R\}$. Note that for any $p\in L_{z_0}$ we have 
$H(p)\in L_0$. Therefore setting $I_{z_0}:=H^{-1}\circ F \circ H$ we get for 
any $p\in L_{z_0}$
\begin{equation*}
 I_{z_0} (p)= H^{-1} \Big(F\big( H(p) \big) \Big) = H^{-1} \big(H(p)   \big)=p.
\end{equation*}
We conclude as before that $L_{z_0}$ is a geodesic line and that $ I_{z_0}$ 
is a reflection about $L_{z_0}$. 

\smallskip 

Now we set $D_y : \{(0, y,0), \ - 2/\sqrt{-\kappa} < y < 2/\sqrt{-\kappa}  \}
\subset \di (2/\sqrt{-\kappa}) \times \{0  \}$. We consider the isometry 
$G(z)=(-\ov z, -t)$. Observe that $G(p)=p$ for any $p \in D_y$. Therefore 
$D_y$ is a geodesic line and, since $G\circ G ={\rm Id}$, we get that $G$ is a 
reflection about $D_y$.

For any $\theta \in [0,\pi)$ we set 
$D_\theta:= 
\{ (se^{i\theta}, 0),\ - 2/\sqrt{-\kappa} < s < 2/\sqrt{-\kappa} \}$, 
thus $D_y=D_{\pi/2}$. Observe that for any $\alpha\in \R$ the map 
$R_\alpha (z,t):= (e^{i\alpha}z, t)$ is an isometry. Since 
$R_{\pi/2 - \theta}(D_\theta)=D_y$, we get that the map 
$G_\theta:=R^{-1}_{\pi/2 - \theta}\circ G\circ R_{\pi/2 - \theta}$ fixes any 
point of $D_\theta$. Thus $D_\theta$ is a geodesic line and $G_\theta$ is a 
reflection about $D_\theta$.
 
 We can prove in the same way that  any horizontal geodesic  admits a reflection.

\medskip 

\item We suppose now $\kappa=0$, therefore 
we have that $\E (0, \tau)= {\rm Nil}_3 (\tau)$ is the Heisenberg group and it 
can be viewed as $\R^3$ with the metric 
\begin{equation*}
 g_\tau=dx^2+dy^2+\left(\tau(ydx-xdy)+dt\right)^2.
\end{equation*}
The isometries of $(\R^3, g_\tau)$ are (setting $z=x+iy$)

\begin{equation}
  F(z,t)=(e^{i\theta} z+a+ib, t +\tau \Im (a-ib)e^{i\theta} z +c)
\end{equation}
and
\begin{equation}
G(z,t)=(e^{i\theta} \ov z+a+ib, -t -\tau \Im (a+ib)e^{-i\theta} z +c)
\end{equation}
where $a,b,c,\theta \in \R$ are any real numbers. 
 Observe again that any isometry is orientation preserving.
 
Arguing as in the case $\kappa<0$, it can be shown  that the Euclidean lines 
$L_{z_0}:=\{(z_0,t),\ t\in\R  \}$ and 
$D_\theta:=\{(s e^{i\theta},0),\ s\in \R  \}$ are geodesic lines which admit a 
reflection, for any 
$z_0=(x_0,y_0)\in \R^2$ and any $\theta\in [0,\pi)$. 

Any horizontal geodesic  admits a reflection.
\end{enumerate}

\medskip

\item At last, it is not difficult to construct  smooth and non analytic 
three 
manifolds having geodesic lines of reflection. For example 
let $M^2$ be a smooth Riemannian 
surface. 
\begin{enumerate}
 \item Let $\gamma\subset M^2$ be a geodesic arc. Assume that there exist an open 
oriented neighborhood 
$V$ of $\gamma$ in $M^2$ and a non trivial isometry $i_\gamma : V\longrightarrow V$ such that 
\begin{itemize}
 \item $i_\gamma$ reverses the orientation,
 \item $i_\gamma (p)=p$ for any $p\in \gamma$,
 \item $i_\gamma \circ i_\gamma = Id_V$.
 \end{itemize}
Then setting $U=V\times \R\subset M^2 \times \R:= M^3$ and $I_\gamma( x,t)= (i_\gamma (x), -t)$, we see that 
$\gamma$ admits also a reflection 
in the  Riemannian product $M^2\times \R$.

\item Now assume that there exist a point $p_0\in M^2$, an open oriented neighborhood 
$V\subset M^2$ of $p_0$ and a non trivial  isometry $\varphi : V\longrightarrow V$ such that 
\begin{itemize}
 \item $\varphi$ is orientation preserving,
 \item $\varphi (p_0)=p_0$,
 \item $\varphi \circ \varphi = Id_V$.
\end{itemize}
Then setting $U=V\times \R\subset M^2 \times \R:= M^3$ and $I_\gamma( x,t)= (\varphi (x), t)$, we see that 
$\gamma:= \{p_0\}\times \R$ admits also a reflection 
in the  Riemannian product $M^2\times \R$.
\end{enumerate}


 \end{enumerate}
\end{example}

\bigskip 

\subsection{Boundary regularity}
$  $

 \begin{definition}\label{D.elliptic.quasi}
 Let $\Omega \subset \R^n$, $n\geq 2$, be a  domain and let $Q$  be 
 a second order quasilinear  operator of the following form 
 \begin{equation*}
 Q(u) := \sum_{i,j=1}^n a^{ij}(x,u,Du) D_{ij}u + b(x,u,Du), 
\end{equation*}
where $x\in \Omega$, and the functions 
$a^{ij}, b$ are defined and $C^1$ on $\ov \Omega \times \R\times \R^n$. 

We assume that $(a^{ij} (x,z,p))_{1\leq i,j \leq n}$ is a symmetric matrix 
for any $(x,z,p)\in \ov \Omega \times \R\times \R^n$. 
We denote by 
$\lambda(x,z,p)$, respectively $\Lambda(x,z,p)$, the minimum eigenvalue, 
respectively the maximum eigenvalue, of 
the symmetric matrix $(a^{ij} (x,z,p))$.

We say that $Q$ is an {\em elliptic operator}, if $0< \lambda(x,z,p)$ for 
any 
$(x,z,p)\in \ov \Omega \times \R\times \R^n$.

Assume now that $\Omega$ is a bounded domain. Then, by  
 continuity, for any $K>0$, there exist 
constant numbers 
$0 < \lambda_K \leq  \Lambda_K$ and $\mu_K >0$ such that 
\begin{align*}
 0 < \lambda_K \leq \lambda(x,z,p) \leq \Lambda(x,z,p) \leq  \Lambda_K, \\
\Big(\lvert a^{ij}\lvert + \lvert D_{p_k}a^{ij}\lvert + \lvert D_z 
a^{ij}\lvert 
 + \lvert D_{x_k} a^{ij}\lvert +  \lvert b\lvert \Big)(x,z,p) \leq \mu_K, 
\end{align*}
 for any $x\in \ov \Omega$ and $(z,p)\in \R\times \R^n$ satisfying 
$\vert z\vert + \vert p\vert \leq K$. 
\end{definition}

A crucial  ingredient in the proof of the Main Theorem is a uniform global 
H\"older 
estimates 
for the gradient for a   solution  of a general second order 
elliptic quasilinear equation. This can be seen as an 
 extension of the well-know Ladyzhenskaya
and Ural'tseva fundamental global  a priori H\"older estimates 
\cite[Chapter IV, Theorem 6.3]{L-U}.

\begin{theorem}\label{T.Holder-Trudinger}
Let $\Omega \subset \R^n$, $n\geq 2$, be a bounded domain with $C^{2}$ 
boundary, 
and let $Q$ be a quasilinear operator as above with 
$a^{ij}, b \in C^1(\ov \Omega \times \R\times \R^n)$.

 Let 
 $u\in C^0(\ov \Omega) \cap C^2(\Omega)$ be a 
function satisfying 
\begin{equation*} 
\begin{cases}
  Q(u)=0 \quad \text{on} \quad \Omega, \\
  u=\varphi \quad \text{on} \quad \partial\Omega,
\end{cases}
\end{equation*}
where $\varphi \in C^{2}(\ov \Omega)$.

Assume there is $K>0$ such that $\vert u\vert + \vert Du\vert \leq K$ on 
$\Omega$. 

\smallskip 

Then, there exists  a constant  
  $\tau \in (0,1)$, such that  
  $u\in C^{1,\tau}(\ov \Omega)$. 
 
More precisely, $u\in C^1(\ov \Omega)$ and there exist positive numbers 
$C=C(n, K, \lambda_{\wh K}, \mu_{\wh K},
\lvert \varphi\lvert_{C^2(\ov\Omega)},\Omega)$ and 
$\tau=\tau(n, K, \lambda_{\wh K}, \mu_{\wh K},
\lvert \varphi\lvert_{C^2(\Omega)}  ) \in (0,1)$ 
such that 
for any $x_1 ,x_2\in \ov \Omega$ we have 
\begin{equation}\label{Holder global}
 \left\lvert Du (x_1) - D u(x_2) \right\lvert \leq C \, 
 \lvert x_1 - x_2\lvert^\tau,
\end{equation}
where 
 $\wh K:= K + \lvert \varphi\lvert_{C^1(\Omega)}$. 
\end{theorem}

\begin{remark}
We observe that the assumption that $u$ has bounded gradient in 
Theorem \ref{T.Holder-Trudinger} is crucial. Indeed, consider in $\R^3$ 
a vertical catenoid $C$. Assume that the {\em neck} of $C$ stays in the 
horizontal plane $\{ x_3=0\}$. Thus the part of $C$ staying between the 
planes  $\{ x_3=0\}$ and $\{ x_3=1\}$ is the graph  of a function $u$ defined 
on an annulus in the plane $\{ x_3=0\}$. The function $u$ satisfies the minimal 
equation on this annulus of $\R^2$ and the gradient of $u$ is not bounded near 
the inner circle of the annulus. Therefore $u$ has not extension  
$C^{1,\tau}$ up to the boundary.
\end{remark}

\ 

\begin{proof}
 
We present  Trudinger's proof, see \cite[Theorem 4]{Trudinger}, omitting the derivation of certain 
assertions, 
but giving further details for the sake of clarity.

\smallskip 
 
 We set $\wh u:= u-\varphi$. Thus 
 $\vert \wh u\vert + \vert D\wh u\vert \leq \wh K$ in $\Omega$, 
 $\wh u=0$ on $\partial\Omega$ 
 and $\wh u$  satisfies 
\begin{equation*}
 \sum_{i,j=1}^n \wh a^{ij}(x,\wh u,\wh Du) D_{ij}\wh u + \wh b(x,\wh u,D\wh u)
 =0
\end{equation*}
 with 
\begin{align*}
 \wh a^{ij}(x,\wh u,\wh Du) &:= a^{ij}(x, \wh u + \varphi, D\wh u + 
D\varphi)\\
\wh b(x,\wh u,D\wh u) &:=b (x, \wh u + \varphi, D\wh u + D\varphi) + 
 \sum_{i,j=1}^n \wh a^{ij}(x,\wh u,\wh Du)D_{ij}\varphi.
\end{align*}

 We consider the linear operator on $\Omega$
\begin{equation*}
 L(\omega)=\sum_{i,j} \alpha^{ij}(x) D_{ij} \omega
\end{equation*}
 for $\omega\in C^2(\Omega)$, where 
 $\alpha^{ij}(x):=\wh a^{ij}(x, \wh u(x), D\wh u (x))$. 
Observe that $(\alpha^{ij}(x))$ is a bounded and continuous symmetric matrix 
on $\Omega$. Furthermore we have  
$0< \lambda_{\wh K} \leq \rho (x) \leq \Lambda_{\wh K}$ for any $x\in \Omega$, 
where 
$\rho (x)$ is any eigenvalue of the matrix $(\alpha^{ij}(x))$.

Observe also that $\wh u$ satisfies $L(\wh u)=f(x)$, where 
$f(x):=-\wh b(x, \wh u(x), D\wh u (x))$.
 
 \medskip 
 
Next we state without proofs some structure facts that we use in the sequel. 

\medskip 

For any vector $d=(d_1,\dots,d_k)\in \R^k$, $k\in \N^*$, we set  
$\vert d\vert=\sqrt{d_1^2 + \cdots + d_k^2}$. 
 \medskip 
 
 Let $p\in \partial \Omega$ be any boundary point. Since $\partial\Omega$ is a 
compact embedded hypersurface of $\R^n$ with $C^2$ regularity, we deduce first 
that there exists a constant $A>0$ such that for any $p\in \partial \Omega$ and 
for any normal curvature $k_n(p)$ of $\partial\Omega$ at $p$, we have 
$\lvert k_n (p)\lvert \leq A$. 

\smallskip 

We deduce also that there exists a positive constant $R<1$, depending only on 
the 
geometry of 
 $\partial \Omega$, and not on $p$, such that 
\begin{itemize}

 \item $\Omega \cap B_R(p)$ is connected, where $B_R(p)$ is the ball centered 
at p with radius $R$.

\item We choose orthonormal coordinates $(y_1,\cdots,y_n)$ in $\R^n$ such that 
$p=0$ in those coordinates and $(\frac{\partial}{\partial y_1}$,\dots,
$\frac{\partial}{\partial y_{n-1}})$ is a basis of the tangent space of 
$\partial\Omega$ at $p$. Then a neighborhood of $p$ in $\partial\Omega$ is the 
graph of a $C^2$ function $h$ defined in a neighborhood of $0$ in 
$\{y_n =0  \}$ containing the {\em disk} $\{\lvert y\lvert < R,\ y_n=0 \}$. 
Moreover $h$ satisfies 
\begin{itemize}
 \item $h(0)=0$ and $Dh (0)=0$,
 
 \item $\lvert Dh\lvert < 1/4$ and 
 $\lvert \frac{\partial^2 h}{\partial y_i \partial y_j}\lvert \leq 8A$ on the 
whole domain where $h$ is defined.
\end{itemize}

\item The map $F : B_R(p)\longrightarrow \R^n$ defined by 
$F(y):= y -(0,\dots,0,h(y_1,\dots,y_{n-1}))$ is 
 a $C^2$ diffeomorphism onto its image, 
satisfying $F (p)=0$ and
\begin{itemize}
 \item $F(\Omega \cap B_R(p)) \subset \{y_n >0  \}$,
 
 \smallskip 
 
 \item $F(\partial\Omega \cap B_R(p)) \subset \{y_n =0 \}$,
 
 \smallskip 
  
 \item  $\lvert D F_i\lvert \leq  2$, 
 where $F=(F_1,\dots,F_n)$,
  
 \smallskip 
 
 \item $\lvert D_{qs} F_i\lvert \leq 8A$, $q,s=1,\dots,n$, 
  
 \smallskip 
 
 \item for any positive $r \leq R$, setting 
 $B_r:=\{y\in \R^n,\ \lvert y\lvert <r \}$ and 
  $B_{r}^+= B_{r} \cap \{y_n >0  \}$, 
 we have 
 $B_{r/2}^+ \subset F(\Omega \cap B_r(p))$ and 
 $B_{r/4}(p) \cap \Omega \subset F^{-1}(B_{r/2}^+)$.
\end{itemize}
\end{itemize}
Thus, the function $w:= \wh u\circ F^{-1}$ is defined on $B_{R/2}^+$ and 
satisfies the linear elliptic equation 
\begin{align}
\label{Eq.Linear} 
 \wt L(w) & := \sum_{i,j}\wt \alpha ^{ij} (y) D_{ij}w = \wt f(y) \quad 
\text{in}\quad B_{R/2}^+, \\
w&= 0 \quad \text{on} \quad \partial B_{R/2}^+ \cap \{y_n =0\},\notag
\end{align}
where
\begin{itemize}
 \item $\wt \alpha ^{ij} (y):= 
 \sum_{q,s}\alpha ^{qs} (F^{-1}(y)) \, D_s F_j\, D_q F_i$,
 
 \item  $\wt f(y):= -\sum_i\Big(\sum_{q,s} \alpha ^{qs} (F^{-1}(y))D_{qs}F_i   
\Big) D_iw + f(F^{-1}(y)) $.
\end{itemize}
Observe that we have $(\wt \alpha ^{ij})=DF\, (\alpha^{ij})\,  ^t DF$. Taking 
into account the definition of $F$, a straightforward computation shows that 
we have 
$\dfrac{\lambda_{\wh K}}{2} \leq \wt \rho (y)\leq 
4 \Lambda_{\wh K}$ for any eigenvalue $\wt \rho (y)$ of the matrix 
$( \wt \alpha ^{ij} (y))$ and that $\lvert w\lvert + \lvert 
Dw\lvert \leq (1 +2\sqrt{n}) \wh K$. 

Since $u$ has bounded gradient we observe that the function $\wt f$ is bounded 
on $B_{R/2}^+$.

\smallskip 

We consider the function $v(y):=\dfrac{w(y)}{y_n}$ on   $B_{R/2}^+$ and we set 
$\delta:=\frac{\lambda_{\wh K}}{48 n\Lambda_{\wh K}}$. 

Since $u$ has bounded gradient we deduce from the 
proof of 
\cite[Theorem 1.2.16]{Han} that there exist real numbers 
$C_1 = C_1(n, K, \lvert \varphi\lvert_{C^1(\Omega)}, \lambda_{\wh K},
\Lambda_{\wh K})>0$ and $\alpha=\alpha(n, K, \lvert 
\varphi\lvert_{C^1(\Omega)}, \lambda_{\wh K},
\Lambda_{\wh K})\in (0,1)$,  such that for any positive $r \leq 
\frac{\delta}{4}\frac{R}{2}$ and 
for any 
$x,y \in B_r^+$ we have
\begin{equation}\label{F.Krylov}
 \left\vert v(x) - v(y)  \right\vert \leq 
 \frac{C_1}{  R^\alpha}r^\alpha \big( 
\sup_{B_{R/2}^+} \vert Dw\vert + 
{  R}\sup_{B_{R/2}^+} \vert \wt f\vert   \big) .
\end{equation}
Observe that in \cite[Theorem 1.2.16]{Han}
it is assumed  that 
$w \in C^1(B_{R/2}^+ \cup \Sigma_{R/2})$ where 
$\Sigma_{R/2}:=B_{R/2} \cap \{y_n=0 \}$. 
Nevertheless, what is required in the proof is that 
$\vert Dw \vert$ is bounded 
on $B_{R/2}^+$. 

It follows from Proposition \ref{P.Holder.boundary} in the Appendix 
 that the function $v$ extends to a continuous 
function on 
$ B_{\delta R/16} ^+ \cup \Sigma_{\delta R/16}$. Moreover we obtain the 
boundary H\"older 
estimates of  Krylov: for any $x^\prime,y'\in \Sigma_{\delta R /64}$ we have 
\begin{equation}\label{v.Holder.boundary}
 \vert v( x')- v( y')\vert \leq \frac{ C_1}{R^\alpha} 
 \big( \sup_{B_{R/2}^+} \vert Dw\vert + R \sup_{B_{R/2}^+} \vert \wt f\vert   
\big) \vert  x'- y'\vert^\alpha .
\end{equation}
Thus the function $w$ admits normal derivative along $\Sigma_{\delta R/32}$. 
Using Formula (\ref{v.Holder.boundary}) it can be shown that $v$ is 
H\"older continuous on $ B_{\delta R/256} ^+ \cup \Sigma_{\delta R/256}$, 
see 
Proposition \ref{P.V.Holder} in the Appendix. 

More 
precisely there exist positive numbers 
$C_2=C_2(n,K,\lvert \varphi\lvert_{C^1(\Omega)}, 
\lambda_{\wh K},\Lambda_{\wh K},R)$ and 
$\beta=\beta(n,K, \lvert \varphi\lvert_{C^1(\Omega)}, 
\lambda_{\wh K},\Lambda_{\wh K})\leq \alpha <1$,  
such that for any 
$x,y\in B_{\delta R/256} ^+ \cup \Sigma_{\delta R/256}$ we have 
\begin{equation}\label{v.Holder}
 \vert v( x)- v( y)\vert \leq \frac{C_2}{R^\beta} 
 \big( \sup_{B_{R/2}^+} \vert Dw\vert + R \sup_{B_{R/2}^+} \vert \wt f\vert   
\big) \vert  x-  y\vert^\beta .
\end{equation}

We know from \cite[Chapter IV, Theorem 6.1]{L-U} and 
\cite[Theorem 13.6]{G-T} that we have  the Ladyzhenskaya and Ural'tseva a 
priori interior 
estimates for the function $\wh u$ on $\Omega$. Namely, there exist positive 
constants 
$C_3, \eta <1$ such that 
for any subdomain 
$\Omega^\prime \subset \ov \Omega^\prime\subset \Omega$, we have 
for any $x_1,x_2 \in \Omega^\prime$ 
\begin{equation}\label{F.Interior Estimates}
 \lvert  D\wh u(x_1) - D\wh u(x_2)\lvert \leq C_3\, \lvert x_1-x_2\lvert^\eta 
\, 
 \dist (\Omega^\prime, \partial\Omega)^{-\eta}
\end{equation}
with  $C_3=C_3(n,K,\mu_{\wh K},\lambda_{\wh K}, 
\lvert \varphi\lvert_{C^2(\Omega)}, \diam (\Omega)$ and 
$\eta=\eta(n,K,\mu_{\wh K},\lambda_{\wh K},
\lvert \varphi\lvert_{C^2(\Omega)} )$. Note that the dependence  of 
$\lvert \varphi\lvert_{C^2(\Omega)}$ arises from the definition of $\wh b$.

\medskip 

Now we extend the function $w$ by odd reflection to the whole ball 
$B_{R/2}$ setting for any $y=(y^\prime, y_n) \in B_{R/2}$
\begin{equation*}
 \ov w (y)=
\begin{cases}
 w(y) \ \ \ &\text{if} \ y_n \geq 0 \\
 -w(y^\prime, -y_n) \ \ \ &\text{if} \ y_n < 0.
\end{cases}
\end{equation*}
Observe that $\ov w$ is a continuous function and that 
$\ov w \in C^2(B_{R/2} \setminus \Sigma_{R/2})$. Consequently, $v$ extends also 
to a continuous function $\wt v$ on the whole ball 
$B_{\delta R/16}$ by setting 
$\wt v (y^\prime, y_n):=\dfrac{\ov w (y^\prime, y_n)}{y_n}$ for any 
$y\in B_{\delta R/16} \setminus\Sigma_{\delta R/16}$ and 
$\wt v (y^\prime,0):= v (y^\prime,0)$ for any 
$(y^\prime,0) \in \Sigma_{\delta R/16}$.

From (\ref{v.Holder}) we get that for any $x,y \in B_{\delta R/256} $ we 
have 
\begin{equation}\label{v.Holder-entiere}
 \vert \wt v( x)- \wt v( y)\vert \leq 2\frac{C_2}{R^\beta} 
 \big( \sup_{B_{R/2}^+} \vert Dw\vert + R \sup_{B_{R/2}^+} \vert \wt f\vert   
\big) \vert  x-  y\vert^\beta .
\end{equation}
We set $R_1:=\delta R/256$.
Using the a priori interior H\"older estimates (\ref{F.Interior Estimates})   
for 
the gradient of $\wh u$  
 and the
H\"older estimates (\ref{v.Holder-entiere}) for $\wt v$, Trudinger derived 
in the proof of \cite[Theorem 4]{Trudinger} that 
for any $x,y\in B_{R_1/4}$ we have 
\begin{equation*}
 \lvert \ov  w (x+y) + \ov  w (x-y) - 2\ov  w (x) \lvert \leq 
 C_4 \lvert y \lvert^{1+\gamma},
\end{equation*}
where $\gamma:=\beta\eta/(1+\eta)$ and 
$C_4=C_4(n,K,\mu_{\wh K}, \lambda_{\wh K},\lvert \varphi\lvert_{C^2(\Omega)}, 
\Omega)$.

\smallskip 

Now consider a $C^\infty$ function $\psi$ on the whole $\R^n$ such that 
\begin{equation*}
\begin{cases}
 \psi (y)= 1 \ \ \ \text{if} \ \ \lvert x\lvert \leq R_1/32 \\
   0 < \psi (x) <1 \ \ \ \text{if} \ \  R_1/32 < \lvert x\lvert < 3R_1/64 \\
 \psi (y)= 0 \ \ \ \text{if} \ \ \lvert x\lvert \geq 3R_1/64.
\end{cases}
 \end{equation*}
Then $\psi \ov  w$ is a continuous function defined on $\R^n$. It can be 
shown that there exists a constant 
$C_5=C_5(n,K,\mu_{\wh K}, \lambda_{\wh K},\lvert \varphi\lvert_{C^2(\Omega)}, 
\Omega)$
 such that 
for any $x,y\in \R^n$ we have 
\begin{equation*}
 \lvert \psi\ov  w (x+y) + \psi\ov  w (x-y) - 2 \psi\ov  w (x) \lvert \leq 
 C_5 \lvert y \lvert^{1+\gamma}.
\end{equation*}
We deduce from \cite[Chapter V, section 4, Propositions 8 and 9]{Stein} that 
$\psi \ov  w\in C^{1,\gamma}( \R^n)$. 
More precisely, $\psi \ov  w\in C^{1}( \R^n)$ and 
there exists a universal constant $\Upsilon >0$ such that 
for any $x,y \in \R^n$ we have
\begin{equation*}
 \lvert D \psi\ov  w (x)-D \psi\ov  w(y)\lvert \leq \Upsilon C_5 
 \lvert x-y\lvert^\gamma.
\end{equation*}
Therefore, 
$w\in C^{1,\gamma} (B_{R_1/32}^+ \cup \Sigma_{R_1/32} )$ and for 
any $x,y\in B_{R_1/32}^+ \cup \Sigma_{R_1/32}$ we have
\begin{equation*}
 \lvert D  w (x)-D   w(y)\lvert \leq \Upsilon C_5 
 \lvert x-y\lvert^\gamma.
\end{equation*}

\smallskip

Recall that $\wh u= w \circ F$ on 
$\Omega \cap  B_{R_1/64}(p)$ and 
$F (\Omega \cap  B_{R_1/64}(p)) \subset B_{R_1/32}^+$. 
Therefore for any $p\in \partial \Omega$, the restriction of 
 $\wh u$ at  $\Omega \cap B_{R_1/64}$ belongs to 
$C^{1,\gamma}(\ov \Omega \cap  B_{R_1/64}(p))$.

More precisely there exist  
positive constants 
$C_6=C_6(n, K, \lambda_{\wh K}, \mu_{\wh K},\lvert 
\varphi\lvert_{C^2(\Omega)}, \Omega) $, and 

$\gamma=\tau(n, K, \lambda_{\wh K}, \mu_{\wh K},
\lvert \varphi\lvert_{C^2(\Omega)}  ) <1$, 
but which do not 
depend on $p\in \partial\Omega$, such that for any 
$x_1,x_2 \in \ov \Omega \cap  B_{R_1/64}(p)$ we have 
\begin{equation}\label{Holder boundary}
 \left\lvert D\wh u (x_1) - D\wh u(x_2) \right\lvert \leq C_6 \, 
 \lvert x_1 - x_2\lvert^\gamma.
\end{equation}
Thus $\wh u\in C^1(\ov \Omega)$. 

Finally, since $\partial\Omega$ is compact, there exist a finite number of 
points $p_1,\dots, p_k \in \partial\Omega$ such that 
$\partial\Omega\subset \bigcup_{i=1}^k \Omega \cap  B_{R_1/128}(p_i) $. 

We set 
$\Omega_0:= \Omega\setminus \bigcup_{i=1}^k \Omega \cap  B_{R_1/128}(p_i)$, 
thus $\Omega =\Omega_0 \cup \left(\bigcup_{i=1}^k \Omega \cap  
B_{R_1/64}(p_i)\right)$.

Considering the interior estimates (\ref{F.Interior Estimates}) for 
$\Omega^\prime = \Omega_0$,  the boundary H\"older estimates 
(\ref{Holder boundary}) at each subset $ \Omega \cap  B_{R_1/64}(p_i)$, 
$i=1,\dots,k$, and a ball chain argument, we conclude that 
$D\wh u \in C^\tau (\ov\Omega)$ where $\tau := \min (\gamma,\eta)$. More 
precisely $\wh u\in C^1(\ov \Omega)$ 
and there exist positive constants 
$C_7(n, K, \lambda_{\wh K}, \mu_{\wh K},
\lvert \varphi\lvert_{C^2(\Omega)}, \Omega) $, and 
$\tau (n, K, \lambda_{\wh K}, \mu_{\wh K},\lvert \varphi\lvert_{C^2(\Omega)} ) 
<1$, such that for any 
$x_1,x_2 \in \ov \Omega $ we have 
\begin{equation*}
 \left\lvert D\wh u (x_1) - D\wh u(x_2) \right\lvert \leq C_7 \, 
 \lvert x_1 - x_2\lvert^\tau. 
\end{equation*}

Since $\varphi \in C^2(\ov \Omega)$, 
there exists a positive constant 
$C_8=C_8(\lvert \varphi\lvert_{C^2(\Omega)}, \Omega, \tau)$ such that 
\begin{equation*}
 \vert D\varphi (x_1) -  D\varphi (x_2)\vert \leq C_8 
 \vert x_1 - x_2 \vert^\tau,
\end{equation*}
for any $x_1,x_2 \in \ov \Omega$. 

Finally, since $u = \wh u + \varphi$, 
setting $C:= C_7 + C_8$,  
we 
have for any 
$x_1,x_2 \in \ov \Omega $ 
\begin{equation*}
 \left\lvert Du (x_1) - D u(x_2) \right\lvert \leq C \, 
 \lvert x_1 - x_2\lvert^\tau,
\end{equation*}
where $C=C(n, K, \lambda_{\wh K}, \mu_{\wh K},
\lvert \varphi\lvert_{C^2(\Omega)}, \Omega)$. Thus we obtain that 
$u\in C^{1,\tau}(\ov \Omega)$ as desired. 
\end{proof}

\bigskip 

We infer from the proof of Theorem \ref{T.Holder-Trudinger} the following local 
version.

\begin{theorem}\label{Holder local boundary}
Let $\Omega \subset \R^n$, $n\geq 2$, be a bounded domain with $C^{2}$ 
boundary, 
and let $Q$ be a quasilinear operator as in Definition 
\ref{D.elliptic.quasi}, with 
$a^{ij}, b \in C^1(\ov \Omega \times \R\times \R^n)$.

Let 
 $T \subset \partial \Omega$ be a nontrivial  $C^{2}$ domain of the boundary 
of $\Omega$.  Let  
$\Omega_0\subset\Omega$ be a  subdomain 
 with $C^{2}$ boundary
such that 
$\ov \Omega_0 \cap \partial \Omega \subset T$.

 Let 
 $u\in C^0(\Omega \cup T) \cap C^2(\Omega)$ be a 
function satisfying 
\begin{equation*} 
\begin{cases}
  Q(u)=0 \quad \text{on} \quad \Omega, \\
  u=\varphi \quad \text{on} \quad T,
\end{cases}
\end{equation*}
where $\varphi \in C^{2}( \Omega \cup T)$.

Assume there is $K>0$ such that $\vert u\vert + \vert Du\vert \leq K$ on 
$\Omega$.

Then, there exists  a constant $\tau \in (0,1)$, such that  
$u\in C^{1,\tau}( \ov \Omega_0 )$

More precisely, $u\in  C^1( \ov\Omega_0) $ and there exist
positive 
numbers $C=C(n, K,  \lambda_{\wh K}, \Lambda_{\wh K}, 
 \lvert\varphi \lvert_{C^2( \ov \Omega_0)}, \Omega_0)$ 
and $\tau = \tau (n, K,  \lambda_{\wh K}, \Lambda_{\wh K}, 
 \lvert\varphi \lvert_{C^2(\ov \Omega_0)})$
 such that
for any 
$x_1,x_2 \in \ov \Omega_0 $ we have 
\begin{equation*}
 \lvert Du (x_1)- Du (x_2)\lvert \leq 
 C \lvert x_1 - x_2\lvert^\tau,
\end{equation*}
where $\wh K:= K + \lvert \varphi\lvert_{C^1(\ov\Omega_0)}$. 
\end{theorem}

\bigskip 
 
 Elliptic regularity leads to the following.
 
\begin{theorem}\label{T.Schauder}
Let $\Omega \subset \R^n$, $n\geq 2$, be a bounded domain with $C^{k+2}$ 
boundary, $k \geq 0$, 
and let $Q$ be a quasilinear operator as above with 
$a^{ij}, b \in C^{k+1}(\ov \Omega \times \R\times \R^n)$.

 Let 
 $u\in C^0(\ov \Omega) \cap C^2(\Omega)$ be a 
function satisfying 
\begin{equation*} 
\begin{cases}
  Q(u)=0 \quad \text{on} \quad \Omega, \\
  u=\varphi \quad \text{on} \quad \partial\Omega,
\end{cases}
\end{equation*}
where $\varphi \in C^{k+2}(\ov \Omega)$.

Assume there exists a constant $K>0$ such that 
$\vert u\vert + \vert Du\vert \leq K$ on 
$\Omega$. 

Then, there exists  a constant   
  $\tau \in (0,1)$, such that  
  $u\in C^{k+1,\tau}(\ov \Omega)$, where \newline 
  $\tau = \tau (n, K,  \lambda_{\wh K}, \mu_{\wh K}, 
 \lvert\varphi \lvert_{C^2( \ov\Omega)})$ and 
 $\wh K:= K + \lvert \varphi\lvert_{C^1(\ov\Omega)}$. 
\end{theorem}

\begin{proof}
 The proof proceeds by induction on $k\geq 0$.

 For $k=0$ this is   Theorem \ref{T.Holder-Trudinger}. 
 
 The rest of the proof is a straightforward  consequence of 
 Schauder theory, see  \cite[Theorem 6.19]{G-T}.
\end{proof}

\begin{remark}\label{R.local}
 There is a local version of Theorem \ref{T.Schauder}. Namely 
let 
 $T \subset \partial \Omega$ be a nontrivial   domain of the boundary 
of $\Omega$. Let us assume that  $\varphi \in C^{k+2}(\Omega \cup T)$.

 Let $\Omega_0 \subset \Omega$ be a  subdomain with $C^{k+2}$ 
boundary such that 
$\ov \Omega_0 \cap \partial\Omega \subset T$. Let 
$u\in C^2(\Omega) \cap C^0(\Omega \cup T)$ satisfying $Q(u)=0$ on $\Omega$ and 
$u=\varphi$ on $T$.

  Then we have 
$u\in C^{k+1,\beta}(\ov \Omega_0 )$, where 
$\beta=\beta (n, K,  \lambda_{\wh K}, \mu_{\wh K}, 
 \lvert\varphi \lvert_{C^2( \ov\Omega_0)})$ and \linebreak
 $\wh K:= K + \lvert \varphi\lvert_{C^1(\ov\Omega_0)}$. 
\end{remark}

\section{Proof of the Main Theorem}

\subsection{Minimal equation}
$  $

We first give the minimal equation for a graph $x_3=u(x_1,x_2)$ 
in some arbitrary  local coordinates 
$(x_1,x_2,x_3)$ of $M$, following 
Colding-Minicozzi
\cite[Equation (7.21)]{C-M} and Gulliver \cite[Section 8]{Gulliver}.

Let $u$ be a $C^\infty$ function defined on a domain  $\Omega$  
contained in the 
$x_1,x_2$ plane of coordinates. Let $S\subset M$ be the graph of $u$.

We use the usual convention for the partial derivative of a $C^2$  
function $u$:
$u_i=\frac{\partial u}{\partial x_i}$, 
$u_{ij}=\frac{\partial^2 u}{\partial x_i \partial x_j}$, $i,j =1,2$.
We denote $g_{ij}\in C^\infty$, $1\leq i,j \leq 3$, the coefficients of the 
Riemannian metric $g$ in the local coordinates $(x_1,x_2,x_3)$ and we call $G$ 
the $3\times 3$ matrix $(g_{ij})$. Up to restricting the local 
coordinates, we can assume that the matrix $G$ is bounded.

Let 
$ \Gamma_{ij}^m\in C^\infty$ be the Christoffel symbols 
of the Riemannian metric $g$, $1\leq i,j,m \leq 3$. We set 
$\partial_i:=\frac{\partial}{\partial x_i}$, $i=1,2,3$. 
Then, $X_i:=\partial_i + u_i \partial_3$, $i=1,2$, is the adapted frame 
field 
generating the tangent plane of $S$.

Let $h_{ij}$ be the coefficients of the metric induced on $S$, that is 
$h_{ij}=g(X_i \,; X_j)$, \linebreak $1\leq i,j\leq 2$.

Let $N:=\sum_i N_i \partial_i$ be the unit normal field on $S$ with 
$N_3 > 0$. We set $W:=1/g(N ; \partial_3)$. We have 
\begin{equation*}
 g(N;\partial_i)= -\frac{u_i}{W},\ i=1,2.
\end{equation*}
Note that  the 
coordinates  of $N$ are given by 
\begin{equation*}
 \begin{pmatrix}
  N_1 \\ N_2 \\ N_3
 \end{pmatrix}=
\frac{1}{W}G^{-1}  
 \begin{pmatrix}
  -u_1 \\ -u_2 \\ 1
 \end{pmatrix}.
\end{equation*}
Since $g(N,N)=1$ we obtain 
\begin{equation*}
 W^2= g^{33} -2\sum_{i=1}^2 u_i g^{i3}  + \sum_{i,j=1}^2 u_i u_j g^{ij},
\end{equation*}
where $G^{-1}= (g^{ij})$ is the inverse matrix of $(g_{ij})$.

The mean curvature $H$ of $S$ is given by 
\begin{equation*}
2H= \sum_{i,j=1}^2 h^{ij} g( N ; \nabla_{X_i}X_j),
\end{equation*}
where $ (h^{ij})$ is the inverse matrix of $(h_{ij})$ and 
$\nabla$ is the covariant derivative on $(M,g)$.

We define 

\begin{align*}
F(x,u,u_1,u_2,u_{11}, u_{12}, u_{22})&:=\sum_{i,j=1}^2 \Big[
h^{ij}\left( u_{ij} + \Gamma_{ij}^3 + u_i\Gamma_{3j}^3 + 
u_j\Gamma_{3i}^3 + u_i u_j\Gamma_{33}^3 \right) \\
&\ \ \ \ \ \ \ \ \ \qquad
-\sum_{m=1}^2u_m h^{ij}\left( \Gamma_{ij}^m + u_i\Gamma_{3j}^m + 
u_j\Gamma_{i3}^m +  u_i u_j\Gamma_{33}^m \right)\Big].
\end{align*}


Then by a computation it follows that the minimal equation  ($H=0$) reads as 
\begin{equation}\label{Eq.Minimal.equation}
 F(x,u,u_1,u_2,u_{11}, u_{12}, u_{22})=0.
\end{equation}

 Since $h_{ij}=g(X_i\, ;\, X_j)$, $(h_{ij})$ is a symmetric and 
positive matrix. This implies that $(h^{ij})$ is also a symmetric and 
positive matrix and, therefore, the 
  equation (\ref{Eq.Minimal.equation}) is an elliptic PDE. 
Furthermore, if $u$ has bounded gradient then the equation 
(\ref{Eq.Minimal.equation}) is uniformly elliptic.  This means that 
there exist two positive constants $\lambda \leq \Lambda$ such that for any 
$x\in \Omega$ and for any eigenvalue $\rho(x)$ of the matrix 
$(h^{ij}(x))$, we have $0 < \lambda \leq \rho (x) \leq \Lambda$.

\begin{remark}\label{R.Analytic}
Let $M$ be an analytic three manifold, and let $S \subset M$ be a  
minimal surface with an analytic open arc $\gamma$ on its boundary. We assume 
that 
$S\cup \gamma$ is a $C^1$ surface with boundary. 

Then for any $p\in \gamma$ there exists a neighborhood $U\subset M$ of $p$ 
in $M$ such that $U\cap S$ is a  graph $x_3= u (x_1,x_2)$ in some local 
coordinates $(x_1,x_2,x_3)$ at $p$, of an analytic function $u$ defined on a 
domain 
$\Omega \subset \{x_3=0  \}$, containing an analytic arc $\gamma_0$ on its 
boundary. Furthermore $\gamma_0$ is the projection of $\gamma$ and 
$u\in C^1(\Omega\cup \gamma_0)$. 

We infer from Theorem \ref{T.Schauder} that 
$u\in C^k(\Omega\cup \gamma_0)$ for any $k\in \N$, and then  
$u\in C^{2,\mu}(\Omega\cup \gamma_0)$ for any $\mu \in (0,1)$. 
We conclude from \cite[Theorem 5.8.6']{Morrey} that $u$ is analytic on 
$\gamma_0$ and extends analytically across $\gamma_0$. Therefore, $S$ can be 
extended analytically across $\gamma$ as a minimal surface.
\end{remark}

\bigskip

\subsection{Proof of Theorem \ref{T.Main.Theorem}}
$  $

Let $p$ be any point on the geodesic arc $\gamma$.  
 We are going to construct convenient local coordinates $(x_1,x_2,x_3)$ of $M$ 
near  $p\in \gamma$.
 
 First we choose a parametrization by arc lenght, 
 $x_1 \in (-\varepsilon, \varepsilon)$, of a open 
subarc of $\gamma$: 
 $\wh \gamma : (-\varepsilon, \varepsilon) \longrightarrow \gamma$,  such that 
$\wh \gamma (0)=p$. 

\smallskip 

Recall that by assumption, $S\cup \gamma$ is an embedded $C^1$ surface with 
boundary. 

Let 
$\nu_p $ be the unit inner tangent vector of $ S\cup \gamma$ at $p$, 
orthogonal to $\gamma$. We denote by $\nu$ the parallel vector field along 
$\gamma$ such that $\nu (p)=\nu_p$.
 By abuse of notation we denote  
 $\nu (x_1)= \nu (\wh \gamma (x_1))$. 
 
 We set 
\begin{equation*}
 \Sigma := \big\{F(x_1,x_2):=\exp_{ \wh \gamma (x_1)} x_2 \nu (x_1),\ x_1,x_2 
\in  (-\varepsilon, \varepsilon)   \big\}.
\end{equation*}
Clearly, if $\varepsilon >0$ is small enough,  then $\Sigma\subset M$ is a 
 properly  embedded $C^\infty$-surface. Furthermore, $\Sigma$ and 
 $S\cup \gamma$ share the same tangent plane at $p$.

 \smallskip

 Let $\eta$ be a $C^\infty$ unit normal vector field along $\Sigma$. Thus if 
 $\varepsilon >0$ is small enough the map 
\begin{align*}
(-\varepsilon, \varepsilon) &\times (-\varepsilon, \varepsilon) \times 
(-\varepsilon, \varepsilon) \longrightarrow M \\
 G(x_1,x_2,x_3)&:= \exp_{F(x_1,x_2)} x_3 \eta,
\end{align*}
is a $C^\infty$ proper embedding. Therefore $G$ provides local coordinates of 
$M$ near $p$, and we have  $G(0,0,0)=p$. We set 
\begin{equation*}
 U_\varepsilon := \big\{(x_1,x_2,x_3),\  
x_1,x_2 , x_3 \in  (-\varepsilon, \varepsilon)  \big\} \subset \R^3 .
\end{equation*}
We define
\begin{equation*}
 V_\varepsilon := G(U_\varepsilon),
\end{equation*}
thus $V_\varepsilon$ is an open neighborhood of $p$ in $M$.

Observe that by construction we have 
\begin{equation*}
 I_\gamma \big( G(x_1,x_2,0) \big)= G(x_1,-x_2,0) \quad \text{and} \quad 
 I_\gamma \big( G(x_1,0,x_3) \big)= G(x_1,0,-x_3) 
\end{equation*}
for any $x_1,x_2,x_3 \in (-\varepsilon, \varepsilon)$. Therefore 
by a continuity argument we get that 
\begin{equation*}
 I_\gamma \big( G(x_1,x_2,x_3) \big)= G(x_1,-x_2, -x_3)
\end{equation*}
for any $x_1,x_2,x_3 \in (-\varepsilon, \varepsilon)$.
 
\medskip 

Now we establish  that  $S$ can be locally extended near $p$ by reflection 
across $\gamma$ as a minimal surface.

By abuse of notations we identify $V_\varepsilon$ with $U_\varepsilon$ 
 and a point $G(x_1,x_2,x_3)$ with 
 $(x_1,x_2,x_3)$. Therefore $\Sigma$ is identified with 
 $(-\varepsilon, \varepsilon)\times (-\varepsilon, \varepsilon)\times \{ 0\}$. 
Observe that the reflection $I_\gamma$ reads as 
\begin{equation*}
 I_\gamma(x_1,x_2,x_3)=(x_1,-x_2,-x_3),
\end{equation*}
 for any $(x_1,x_2,x_3)\in U_\varepsilon$.

Note that  by construction we have $T_p \Sigma = T_p (S\cup \gamma)$. Therefore 
 if 
$\varepsilon >0$ is small enough then 
$G^{-1}\big((S\cup \gamma) \cap V_\varepsilon \big)$ is the graph 
$x_3 =u(x_1,x_2)$ 
of a 
$C^1$ function $u$ defined on  
$\Sigma^+ :=\{(x_1,x_2),\ x_1\in  (-\varepsilon, \varepsilon), x_2 \in 
 [0, \varepsilon) \}$, using the identification 
 $V_\varepsilon \equiv U_\varepsilon $.
 
 Observe that by assumption $u$ is $C^2$ on 
$ (-\varepsilon, \varepsilon) \times (0, \varepsilon)$ and satisfies the 
minimal equation (\ref{Eq.Minimal.equation}).

Since by assumption $u$ is $C^1$ on $\Sigma^+ $, we get that $u$  has bounded 
gradient. 
Hence from Remark \ref{R.local} we obtain  
 that $u$ is $C^k$ on $\Sigma^+ $ for any $k\in \N$.
 
\smallskip 

We define a function $v$ on 
$\Sigma^- :=\{(x_1,x_2),\ x_1\in  (-\varepsilon, \varepsilon), x_2 \in 
 (- \varepsilon, 0] \}$ setting 
\begin{equation*}
 v(x_1,x_2):= -u(x_1,-x_2).
\end{equation*}
By construction $u(x_1,0)=v(x_1,0) =0$ for any $x_1$ and 
$u_i (0,0)=0=v_i(0,0)$, $i=1,2$.

Then we define a function $w$ on 
$\Sigma$, identified with 
$\{(x_1,x_2),\ x_1, x_2\in  (-\varepsilon, \varepsilon)\}$, setting 
\begin{equation*}
 w(x_1,x_2):=
\begin{cases}
 u(x_1,x_2)\  \text{if} \ (x_1,x_2) \in \Sigma^+ \\
 v(x_1,x_2)\  \text{if} \ (x_1,x_2) \in \Sigma^-
 \end{cases}
 \end{equation*}
We have that $v$ is $C^2$ on $\Sigma^- $ and 
 $w$ is $C^1$ on $\Sigma$.
 
In order to check that $w$ is $C^2$ on $\Sigma$ it is enough to prove that the 
partial derivatives up to the second order of $u$ and $v$ agree along the arc 
$\Sigma^+ \cap \Sigma^- =\{(x_1,0),\ x_1\in  (-\varepsilon, \varepsilon)\}$.

For any $(x_1,x_2) \in \Sigma^-$ we have
\begin{equation*}
 v_1 (x_1,x_2) = - u_1 (x_1, -x_2) ,\ \ 
 v_2(x_1,x_2) =  u_2 (x_1, -x_2)
\end{equation*}
and
\begin{equation*}
 v_{11}(x_1,x_2)= -u_{11}(x_1,-x_2),\ \ 
  v_{12}(x_1,x_2)=u_{12}(x_1,-x_2),\ \ 
   v_{22}(x_1,x_2)= -u_{22}(x_1,-x_2).
\end{equation*}
Note also that $u_{ij}(0,0)=0$, $i,j=1,2$.

Therefore we deduce from the previous identities that 
along the geodesic line $\{x_2 =0\}$ we have 
\begin{equation*}
 F(x,u,u_1,u_2,u_{11}, u_{12}, u_{22})-F(x,v,v_1,v_2,v_{11}, v_{12}, v_{22}) =
 h^{22}(u_{22}(x_1,0) - v_{22}(x_1,0)).
\end{equation*}
Since $u$ and $v$ satisfy the minimal equation (\ref{Eq.Minimal.equation}) we 
infer 
\begin{equation*}
u_{22}(x_1,0) = v_{22}(x_1,0)=0
\end{equation*}
for any $x_1\in (-\varepsilon,\varepsilon)$. Therefore 
$u_{ij}=v_{ij}$ along the arc $\Sigma^+ \cap \Sigma^-$.

Thus we deduce that the 
function $w$ is $C^2$ on the whole domain $\Sigma$, satisfying the 
minimal equation (\ref{Eq.Minimal.equation}).  Thereby, the graph of $w$, 
denoted 
by $\wt S$, is a minimal surface. 
Of course, observe that, by 
construction,  $\wt S$ is invariant by reflection 
across $\gamma$.

\smallskip

We obtain therefore a  minimal 
continuation, $\wt S$,  of $S$ across $\gamma$, embedded near $\gamma$.  
This accomplishes the 
proof of the theorem. \qed

\bigskip 

\noindent {\em Proof of the Main Theorem bis.}
By assumption, $f$ is defined on a domain $ \Omega$ in the 
coordinates plane $\{ x_3=0\}$. Moreover the boundary $\partial\Omega$ 
contains a $C^\infty$ open arc $\gamma_0$ such that $\gamma$ is the graph of 
$f$ over $\gamma_0$. 

Since, by assumption, $f$ has bounded gradient,  the proof follows readily 
using the regularity Theorem 
\ref{T.Holder-Trudinger} as in the proof of the Main Theorem 
\ref{T.Main.Theorem}.  \qed

\bigskip 
 
At last, we discuss several general remarks about the geometry of minimal 
surfaces. 

\bigskip 

\begin{remark}
 We use the notations of the Main Theorem \ref{T.Main.Theorem bis} bis.
 
 Assume that the coordinates system $(x_1,x_2,x_3)$ containing the minimal 
surface $S$ has 
the following property: For certain (small enough) domain $\Omega$ contained in 
the 
coordinates plane $\{ x_3=0\}$, we can uniquely solve the Dirichlet Problem for 
the 
minimal equation given any (small enough) continuous data on $\partial\Omega$.

Then we can drop the assumption that $f$ has bounded gradient in 
the Main Theorem bis \ref{T.Main.Theorem bis}. For example this property occurs 
in $\R^3$, $\hd\times \R$ (see \cite[Lemma 3.6]{SE-T6}), and 
$Nil_3$ (see \cite[Theorem 4.1 and Corollary 4.3]{N-E-T}). 
\end{remark}

\begin{remark}
 Let us consider the particular case where the ambient space $M$ is analytic. 
 Let $S\subset M$ be an embedded minimal surface such that 
$S\cup \gamma$ is a $C^1$ surface with 
boundary, where $\gamma$ is an open geodesic arc of $M$ which admits a 
reflection.

Since $\gamma$ is an analytic arc, $S\cup \gamma$ is analytic and can be 
extended as an analytic surface $\wh S$  
across $\gamma$, see Remark \ref{R.Analytic}

Now we can apply either Theorem 1 in \cite{Le} or Theorem 
\ref{T.Main.Theorem}
to infer that in a neighborhood 
of any point of $\gamma$, the extended surface $\wh S$ is invariant by 
reflection across $\gamma$.
\end{remark}
 
The following result have been applied in homogeneous three spaces by many authors 
to construct complete minimal surfaces, see for example \cite{Ros}, \cite{A-R}, \cite{Men}, 
\cite{Ma-Rod-Ros}.

\begin{proposition}\label{P.Plateau}
  Let $(M,g)$ be a $C^\infty$ Riemannian three manifold and let $\Gamma$ be a Jordan 
curve. Assume that $\Gamma$ contains an open geodesic arc
$\gamma$ which admits a reflection. Let $S$ be an area minimizing solution of 
the Plateau problem, if any.

 Then, as in Theorem \ref{T.Main.Theorem}, the reflection of $S$ about $\gamma$ gives rise to a $C^\infty$ 
continuation of 
$S$ across $\gamma$. That is, $S\cup \gamma \cup I_\gamma(S)$ is a smooth  
immersed minimal surface 
which is embedded near $\gamma$. 
\end{proposition}

\begin{proof}

We set 
\begin{equation*}
 B:= \{ (x,y)\in \R^2,\ x^2 + y^2 < 1\} \qquad \text{and} \quad 
 \ov B :=\{ (x,y)\in \R^2,\ x^2 + y^2 \leq 1\}. 
\end{equation*}

By assumption, there exists a  map $X\in C^0(\ov B, M) \cap C^2(B,M)$ such that 
\begin{itemize}
 \item $X(\ov B)=S$, and $X$ maps monotonically $\partial  B$ onto $\Gamma$.  

 \item  $g(\partial_x X ; \partial_x X) -g(\partial_yX ; \partial_y X) = g(\partial_xX ; \partial_y X)=0$, 

\item  
$ \partial_{xx} x_j +  \partial_{yy} x_j + \sum_{k,l=1}^3 \Gamma_{k\, l}^j 
 \big( \partial_x x_k \partial_x x_l + \partial_y x_k \partial_y x_l \big)=0$, 
 $j=1,2,3$, 
 
\item $X$ minimize the functional area  among all maps verifying the above properties, 
\end{itemize}
where $\Gamma_{k\, l}^j$  stands for the  Christoffel symbols, 
$\partial_x =\frac{\partial}{\partial x}$ and so on.

Let $\gamma_0\subset \partial B$ be an open arc 
such that $X(\gamma_0) =\gamma$.

Since $\gamma$ is of class $C^3$, at least, 
we get from \cite[Theorem 4-(ii), Section 2.3]{D-H-T} that 
$X\in C^2(B\cup \gamma_0, M)$. Since $X$ is area minimizing we know from 
\cite[Theorem 8.1]{Gulliver} that $X$ has no {\em true branch point} on $B$. We get from 
 \cite[Conclusion (ii) of Theorem]{He-Hi} that $X$ 
has, possibly, isolated 
boundary branch points on $\gamma_0$ and has no branch point (neither true nor false) in a neighborhood of $\gamma_0$ in $B$. 
We are going to prove that $X$ has no boundary 
branch point on $\gamma_0$.

\smallskip 
 Let $p\in \gamma$ and $p_0 \in \gamma_0$ be such that 
 $p=X(p_0)$. Up to a conformal transformation we can assume that 
$D^+:=\{(u,v)\in \R^2,\ u^2 + v^2 <1,\, \text{and } v \geq 0\}$ is a neighborhood of  
$p_0$ in $B\cup \gamma _0$,  with $p_0$ corresponding to $0$ and 
$\gamma_0$ to $\delta_0:=\{(u,0), \lvert u \lvert <1\}\subset \partial D^+$. 
\smallskip

 By an abuse of notation we assume also that 
$X$ is defined on $D^+$. Furthermore, up to reducing the neighborhood of $p_0$, we can assume that $X$ has no 
boundary branch point on $\delta_0$ except, possibly, at $0$ and has no branch point on 
$D^+ \setminus \delta_0$.  Thus $X: D^+\setminus \delta_0 \longrightarrow M$ is a $C^2$ minimal immersion, 
$X(\delta_0) = \gamma$, $X(0)=p$ and $X$ is of class $C^2$ up to $\delta_0$.

\smallskip

We deduce from  \cite[Conclusion (ii) of Theorem]{He-Hi} that the surface $S$ has a well-defined tangent plane 
at any point of $\gamma$, even at boundary branch point. 

Let $\nu_p$ be unit tangent vector of $S$ at $p\in \gamma$. We consider the coordinates 
$(x_1,x_2,x_3)\in (-\varepsilon, \varepsilon)^3$ on a neighborhood of $p$ in $M$ constructed in the 
proof of Theorem \ref{T.Main.Theorem}. Thus $p$ corresponds  to $(0,0,0)$, $\gamma$ corresponds to $x_2=x_3=0$, 
and the reflection about $\gamma$ (which is assumed to exist) 
reads as $I_\gamma(x_1,x_2,x_3)= (x_1,-x_2, -x_3)$. 

We set 
\begin{equation*}
 D^-:=\{(u,v)\in \R^2,\ u^2 + v^2 <1,\,  v\leq 0\} \ \text{and} \  
D= D^+ \cup D^- =  \{(u,v)\in \R^2,\ u^2 + v^2 <1\}.
\end{equation*}
We are going to show that $X$ can be extended to a $C^2$ map $Z: D\longrightarrow M$ verifying
\begin{itemize}
 \item $Z=X$ on $D^+$,
 \item $g(\partial_u Z; \partial_u Z)- g(\partial_v Z; \partial_v Z)=g(\partial_u Z; \partial_v Z)=0$,
 
 \item $ \partial_{uu} z_j +  \partial_{vv} z_j + \sum_{k,l=1}^3 \Gamma_{k\, l}^j 
 \big( \partial_u z_k \partial_u z_l + \partial_v z_k \partial_v z_l \big)=0$, 
 $j=1,2,3$ on $D$.
 
\end{itemize}
Then using \cite[Theorem]{Gul-Les} we obtain that $0$ is not a branch point of $Z$ and 
consequently, $0$ is not a boundary branch point of $X$. Using Theorem \ref{T.Main.Theorem} we can deduce that 
 $S\cup \gamma \cup I_\gamma(S)$ is a smooth  
immersed minimal surface 
which is embedded near $\gamma$ as desired.

\smallskip 

Recall that we can assume that $X$ has no branch 
point on $\delta_0$ except, possibly, at $0$.

We define the map $Y: D^- \longrightarrow M$  setting for any $(u,v)\in D^-$ : 
$Y(u,v):=I_\gamma (X(u,-v))= (x_1,-x_2,-x_3)(u,-v)$.

Then we define the continuous map $Z: D \longrightarrow M$ setting 
\begin{equation*}
 Z(u,v)=
 \begin{cases}
  X(u,v),\ \text{if}\ v\geq 0, \\
  Y(u,v),\ \text{if}\ v\leq 0.
 \end{cases}
\end{equation*}
We have $y_1 (u,v)=x_1(u,-v)$, $y_2 (u,v)= -x_2(u,-v)$ and 
$y_3 (u,v)= -x_3(u,-v)$ on $D^-$.

We want to show that partial derivatives of $X$ and $Y$ are equal along $\delta_0$ up to second order, 
this will prove that $Z$ is $C^2$ on $D$.

\smallskip

Since $X(\delta_0)=\gamma$ we have $x_2(u,0)=x_3(u,0)=0$ along $\delta_0$.  So that 
$\partial_u x_k (u,0)=\partial_u y_k (u,0)=0$ along $\delta_0$, $k=2,3$. Since
$\partial_v y_k (u,v) = \partial_v x_k (u,-v)$ on $D^-$, $k=2,3$, we deduce that 
$z_2$ and $z_3$ are $C^1$ on $D$. 

By construction we have 
$\partial_u x_1 (u,0)=\partial_u y_1 (u,0)$ along $\delta_0$. 

Since $X$ is a conformal map up to $\delta_0$, we have 
$\partial_u x_1 \cdot \partial_v x_1 =0$ along $\delta_0$. Using the fact that 
$\partial_u X$ does not vanish along $\delta_0$, except possibly at $0$, we get that 
$\partial_v x_1 =0$ along $\delta_0$. Since 
$\partial_v y_1 (u,v)=- \partial_v x_1 (u,-v)$ on $D^-$, we get that 
 $\partial_v x_1 (u,0)=\partial_v y_1 (u,0)=0$ along $\delta_0$ and therefore $Z$ is a 
$C^1$ map on $D$.

\smallskip 

By construction for $k=2,3$ we have 
$\partial_{uv} y_k (u,v)=\partial_{uv} x_k (u,-v)$ on $D^-$, so that 
$\partial_{uv} z_2$ and $\partial_{uv} z_3$ are well 
defined and continuous on $D$.

Since $\partial_v x_1=0$ along $\delta_0$, we have 
 $\partial_{uv} x_1(u,0)=0$ too for any $u$. With the relation 
 $\partial_{uv} y_1 (u,v) = -\partial_{uv} x_1 (u,-v)$ on $D^-$ we get that 
 $\partial_{uv} z_1$ is well 
defined and continuous on $D$.
 
 Furthermore we get by construction that  $\partial_{uu} z_2$,  $\partial_{uu} z_3$, 
 $\partial_{uu} z_1$ and $\partial_{vv} z_1$ 
 are continuous on $D$, so that $z_1$ is a $C^2$ function.
 
It remains to show that  that $\partial_{vv} z_2$ and  $\partial_{vv} z_3$ are well 
defined and continuous on $D$.

Since $X$ is a minimal and conformal immersion  of $D^+\setminus \delta_0$ and $X$ has 
$C^2$ regularity up to $\delta_0$, we have on $D^+$, for $j=2,3$:
\begin{equation*}
 \partial_{uu} x_j +  \partial_{vv} x_j + \sum_{k,l=1}^3 \Gamma_{k\, l}^j 
 \big( \partial_u x_k \partial_u x_l + \partial_v x_k \partial_v x_l \big)=0,
\end{equation*}
where the Christoffel symbols are evaluated at $(x_1,x_2,x_3) (u,v)$, $(u,v) \in D^+$.
 
 Using the previous considerations, we get for $j=2$ and for any $(u,0) \in \delta_0$:
\begin{equation*}
 \partial_{vv} x_2(u,0) + \Gamma_{11}^2 \partial_u x_1 (u,0)^2 + 
 \Gamma_{22}^2 \partial_v x_2 (u,0)^2 +  \Gamma_{33}^2 \partial_v x_3 (u,0)^2 + 
 2  \Gamma_{23}^2 \partial_v x_2 \partial_v x_3 (u,0)=0, 
\end{equation*}
now the  Christoffel symbols are evaluated at $\big(x_1 (u,0),0,0\big)$.

 By the way the coordinates $x_1, x_2$ and $x_3$ are chosen, we have 
 $\Gamma_{kk}^2 (x_1,0,0)=0$ for any $x_1$, $k=1,2,3$.
 
 Moreover a computation shows that 
 $\Gamma_{23}^2 =\frac{1}{2g_{22}} \partial_{x_3}g_{22}$ on $\gamma$. Since $I_\gamma$ is an isometry 
for the metric $g$ we have $g_{22}(x_1,-x_2,-x_3)=g_{22}(x_1,x_2,x_3)$ so that 
$\partial_{x_3}g_{22}(x_1,0,0)=0$ for any $x_1$.

Thus $\partial_{vv} x_2(u,0)=0$ along $\delta_0$, we deduce that  
$\partial_{vv} y_2(u,0)=0$ also and then $\partial_{vv} z_2$ is a well defined and 
continuous map on $D$. This shows that $z_2$ is a $C^2$ map on $D$.

 We prove in the same way that $z_3$ is a $C^2$ map on $D$. This achieves the prove. 
\end{proof}

We write  now a typical example in $\hd \times \R$, see 
\cite[Corollary 4.1]{SE-T5}.

\begin{example}
 Let $T\subset \hd$ be a geodesic triangle with sides $A,B$ and $C$. We assign 
constant value $a,b, c$ respectively on 
${\rm interior} (A), {\rm interior} (B)$, ${\rm interior} (C)$. 

We solve the corresponding Dirichlet problem for the vertical minimal equation 
as in 
\cite[Corollary 4.1]{SE-T5}. We call $f$ the solution and $S$ the graph of $f$. 
Thus, $S$ is a minimal surface of $\hd \times \R$.

It is a matter of fact that the boundary $\Gamma$ of  $S$ 
is constituted of the union of three horizontal segments and  
three vertical segments. It turns out that $S$ is the unique minimal surface 
having 
$\Gamma$ as boundary. Therefore $S$ is the solution of the Plateau problem 
for the boundary data $\Gamma$.

Henceforth, by Proposition \ref{P.Plateau}, we can extend $S$ as a minimal surface 
by 
reflection across any horizontal or vertical lines of $\Gamma$.

\smallskip 

We refer to \cite[Example 4.4]{SE-T5} for a simple construction of a 
complete minimal surface of $\hd\times \R$, by solving a certain Dirichlet 
problem and using reflections about horizontal geodesics. The readers are  
also referred to \cite{SE-T7}.

\end{example}

\bigskip 

\section{Appendix}
 We recall some notations.

 For any vector $d=(d_1,\dots,d_k)\in \R^k$, $k\in \N^*$, we set  
$\vert d\vert=\sqrt{d_1^2 + \cdots + d_k^2}$. 

 We identify $\R^{n-1}$ with $ \{x\in \R^n,\ x_n=0 \}$, that is with 
$\R^{n-1} \times \{0 \}$. Therefore we identify any $ y^\prime \in \R^{n-1}$ 
with 
$(y^\prime , 0)\in \R^n$. 

We note also for any $x\in \R^n$ : 
$x=(x^\prime, x_n)$ where $x^\prime \in \R^{n-1}$ and $x_n \in \R$.

For any $R>0$ and for any $ y^\prime \in \R^{n-1} $, 
we set 
\begin{align*}
 B_R^+ & := B_R \cap \{x_n >0  \} = 
 \{x\in \R^n,\ \vert x\vert < R,\, x_n >0 \} \\
\Sigma_R &:= B_R \cap \{x_n =0  \} = 
\{x\in \R^n,\ \vert x\vert < R,\, x_n =0 \} \\
 B_R^+ ( y^\prime) & := \{x=( x^\prime , x_n)\in \R^n,\ \vert x -  
y^\prime\vert 
< R,\ 
x_n>0  \} \\
 \Sigma_R ( y^\prime) &:=  B_R^+ ( y^\prime) \cap \{x_n =0  \} .
\end{align*}

We recall that $w\in C^2(B_{R/2}^+) \cap C^0(B_{R/2}^+ \cup \Sigma_{R/2})$ 
satisfies the linear elliptic equation (\ref{Eq.Linear}). 
Furthermore $w\equiv 0$ on $\Sigma_{R/2}$ and the function 
$v(y):= \dfrac{w(y)}{y_n}$ defined on $B_{R/2}^+$ satisfies the estimate  
(\ref{F.Krylov}). We set 
$\delta:=\frac{\lambda_{\wh K}}{48 n\Lambda_{\wh K}}$, see the proof of 
Theorem \ref{T.Holder-Trudinger}.

\begin{proposition}\label{P.Holder.boundary}
The function $v$ can 
be extended to a continuous function on \linebreak
$B_{\delta R/16}^+ \cup \Sigma_{\delta R/16}$. 
Moreover, for any $ x^\prime,  y^\prime \in \Sigma_{\delta R/64}$ we have 
\begin{equation*}
 \vert v( x^\prime)- v( y^\prime)\vert \leq \frac{ C_1}{R^\alpha} 
 \big( \sup_{B_{R/2}^+} \vert Dw\vert + R \sup_{B_{R/2}^+} \vert \wt f\vert   
\big) \vert  x^\prime-  y^\prime\vert^\alpha ,
\end{equation*}
where $C_1>0$ and $\alpha \in (0,1)$ are the constants  given in the 
inequality (\ref{F.Krylov}), that is 
$C_1:= C_1(n, K, \lvert \varphi\lvert_{C^1(\Omega)}, \lambda_{\wh K},
\Lambda_{\wh K})$ and 
$\alpha=\alpha(n, K, \lvert \varphi\lvert_{C^1(\Omega)}, \lambda_{\wh K},
\Lambda_{\wh K})$.
\end{proposition}

\begin{proof}
 First observe that for any $y=( y^\prime, y_n)\in B_{\delta R/16}^+$ we 
have 
\begin{equation*}
 B_{y_n}^+( y^\prime) \subset  B_{\frac{\delta}{4}\frac{R}{4}}^+( y^\prime) 
\quad \text{and} \quad  B_{\frac{R}{4}}^+( y^\prime)  \subset B_{R/2}^+. 
\end{equation*}
 Consequently, from  the inequality (\ref{F.Krylov})  applied to the function 
$w$ on 
the half-balls  
 $B_{y_n}^+( y^\prime)$ and $B_{R/4}( y^\prime)$, we deduce that for any 
 $x \in   B_{y_n}^+( y^\prime)$ we have 
\begin{equation}\label{Eq.Estimates}
 \vert v(x)- v(y) \vert \leq \frac{ C_1}{R^\alpha}\, y_n^\alpha 
  \big( \sup_{B_{R/2}^+} \vert Dw\vert + R \sup_{B_{R/2}^+} \vert \wt f\vert   
\big).
\end{equation}

Let $(t_k)$ be a non increasing sequence  of positive real numbers converging 
to $0$. For any 
$p \leq q \in \N$ large enough we deduce from (\ref{Eq.Estimates}) 
(applied to the half-ball $B_{t_q}^+(y^\prime)$), 
that 
\begin{equation*}
 \vert v( y^\prime, t_p) - v( y^\prime, t_q)\vert \leq \frac{C_1}{R^\alpha} 
 t_q ^\alpha 
\big( \sup_{B_{R/2}^+} \vert Dw\vert + R \sup_{B_{R/2}^+} \vert \wt f\vert   
\big).
\end{equation*}
Therefore $(v( y^\prime, t_k))$ is a Cauchy sequence. Consequently 
the sequence 
$(v( y^\prime, t_k))$  converges to some real number, momentarily denoted by  
$h( y^\prime)$. Moreover, the above inequality shows also that the limit 
$h( y^\prime)$ 
does not depend on the positive sequence $(t_k)$ converging to $0$.  

Now let  
$(x_k)=\big( ( x_k^\prime, x_{k,n})\big)$ be a sequence in $B_{\delta R/16}^+$ 
converging 
to $( y^\prime,0)$. We set $\delta_k := \vert x_k - ( y^\prime,0)\vert$. Since 
$x_k\in B_{2\delta_k}^+( y^\prime)$, we deduce from 
(\ref{Eq.Estimates}) that for  $k$ large enough we have 
\begin{align*}
 \vert v(x_k) - h( y^\prime)\vert &\leq \vert v(x_k) - v( y^\prime, 
2\delta_k)\vert + 
 \vert v( y^\prime, 2\delta_k) - h( y^\prime)\vert \\
 &\leq \frac{C_1}{R^\alpha}\, (2\delta_k)^\alpha\, 
 \big( \sup_{B_{R/2}^+} \vert Dw\vert + R \sup_{B_{R/2}^+} \vert \wt f\vert   
\big) + 
   \vert v( y^\prime, 2\delta_k) - h( y^\prime)\vert .
\end{align*}
Therefore we have that $v(x_k) \to h( y^\prime)$. Thus we can extend 
$v$ to a continuous function on $B_{\delta R/16}^+ 
\cup \Sigma_{\delta R/16}$. 

\medskip

Consider now $ x^\prime,  y^\prime \in \Sigma_{\delta R/64}$. Observe 
that 
\begin{equation*}
 B_{\vert  x^\prime- y^\prime\vert}^+ ( y^\prime)
 \subset B_{\frac{\delta}{4} \frac{R}{4}}^+( y^\prime) \quad \text{and} 
\quad 
 B_{R/4}^+( y^\prime)  \subset B_{R/2}^+.
\end{equation*}
Therefore, we get from the inequality (\ref{F.Krylov})  that for any 
$z,z^\prime \in B_{\vert  x^\prime- y^\prime\vert}^+( y^\prime)$ we have 
\begin{equation*}
 \vert v(z) - v(z^\prime)\vert \leq \frac{C_1}{R^\alpha}\, 
 \vert  x^\prime- y^\prime\vert^\alpha 
 \big( \sup_{B_{R/2}^+} \vert Dw\vert + R \sup_{B_{R/2}^+} \vert \wt f\vert   
\big).
\end{equation*}
Now let $( x_k^\prime)$ be any sequence in 
$\Sigma_{\vert  x^\prime- y^\prime\vert}( y^\prime)$ 
 converging to $ x^\prime$ and let $(t_k)$ be a sequence of positive real 
numbers 
converging to 0 such that 
$( x_k^\prime, t_k), ( y^\prime, t_k) \in B_{\vert  x^\prime- 
y^\prime\vert}^+ ( y^\prime)$ for 
any $k$. Therefore, we deduce from the previous inequality that for any $k$ we 
have
\begin{equation*}
 \vert v( x_k^\prime, t_k) - v( y^\prime, t_k)\vert \leq 
 \frac{ C_1}{R^\alpha}\, \vert x^\prime -y^\prime \vert^\alpha 
 \big( \sup_{B_{R/2}^+} \vert Dw\vert + R \sup_{B_{R/2}^+} \vert \wt f\vert   
\big).
\end{equation*}
Letting $k$ go to $+\infty$ we get
\begin{equation*}
 \vert v(x^\prime )- v(y^\prime )\vert \leq \frac{C_1}{R^\alpha} 
 \vert x^\prime - y^\prime \vert^\alpha 
 \big( \sup_{B_{R/2}^+} \vert Dw\vert + R \sup_{B_{R/2}^+} \vert \wt f\vert   
\big),
\end{equation*}
as desired.
\end{proof}
 
\bigskip

For the next result we recall that $\varphi, K$ and $\wh K$ are defined in 
Theorem \ref{T.Holder-Trudinger} and that $\lambda_{\wh K}$ and 
$\Lambda_{\wh K}$ are 
defined in Definition \ref{D.elliptic.quasi}.

We note also that, since $w$ satisfies the linear elliptic equation 
(\ref{Eq.Linear}), the function $v(y)=\dfrac{w(y)}{y_n}$ satisfies on 
$B_{R/2}^+$ the linear elliptic equation 
\begin{equation}\label{Eq.elliptic-2}
\wt L_0 (v) := y_n \sum_{i,j=1}^n \wt \alpha^{ij}(y) D_{ij} v + 2 \sum_{i=1}^n 
\wt \alpha^{in}(y) 
D_i v =\wt f(y)
\end{equation}

\begin{proposition}\label{P.V.Holder}
There exist positive numbers 
$C_2=C_2(n,K,\lvert \varphi\lvert_{C^1(\Omega)}, 
\lambda_{\wh K},\Lambda_{\wh K},R)$ 
 and 
$\beta=\beta(n,K, \lvert \varphi\lvert_{C^1(\Omega)}, 
\lambda_{\wh K},\Lambda_{\wh K})\leq \alpha < 1$,  
such that for any 
$x,y\in B_{\delta R/256} ^+ \cup \Sigma_{\delta R/256}$ we have 
\begin{equation}\label{F.Holder}
 \vert v( x)- v( y)\vert \leq \frac{C_2}{R^\beta} 
 \big( \sup_{B_{R/2}^+} \vert Dw\vert + R \sup_{B_{R/2}^+} \vert \wt f\vert   
\big) \vert  x-  y\vert^\beta .
\end{equation}
\end{proposition}

\begin{proof}
We are going to consider successively the cases 
$x,y \in \Sigma_{\delta R/256}$, 
$x\in \Sigma_{\delta R/256}$, $y\in  B_{\delta R/256}^+$  and 
$x,y\in  B_{\delta R/256}^+$.

\medskip 

\noindent {\bf Case $x,y \in \Sigma_{\delta R/256}$.}

We identify  $x= (x^\prime ,0)$ with $x^\prime $ and $y= (x^\prime ,0)$ with 
$y^\prime $.

Thanks to Proposition \ref{P.Holder.boundary} we have 
\begin{equation*}
 \vert v( x^\prime)- v( y^\prime)\vert \leq \frac{ C_1}{R^\alpha} 
 \big( \sup_{B_{R/2}^+} \vert Dw\vert + R \sup_{B_{R/2}^+} \vert \wt f\vert   
\big) \vert  x^\prime-  y^\prime\vert^\alpha ,
\end{equation*}

\medskip 

\noindent {\bf Case $x\in \Sigma_{\delta R/256}$ and 
$y\in  B_{\delta R/256}^+$.}

We identify  $x= (x^\prime ,0)$ with $x^\prime $. Using inequality 
(\ref{Eq.Estimates}) 
and 
Proposition \ref{P.Holder.boundary} we have 
\begin{align*}
 \vert v(x^\prime ) - v(y)\vert & \leq \vert v(x^\prime ) - v(y^\prime )\vert + 
 \vert v(y^\prime ) -v(y)\vert \\
  &\leq \frac{ C_1}{R^\alpha} 
 \big( \sup_{B_{R/2}^+} \vert Dw\vert + R \sup_{B_{R/2}^+} \vert \wt f\vert   
\big) \vert  x^\prime-  y^\prime\vert^\alpha + 
\frac{ C_1}{R^\alpha} \, y_n^\alpha 
 \big( \sup_{B_{R/2}^+} \vert Dw\vert + R \sup_{B_{R/2}^+} \vert \wt f\vert   
\big)  \\
   &\leq 2\frac{ C_1}{R^\alpha} 
 \big( \sup_{B_{R/2}^+} \vert Dw\vert + R \sup_{B_{R/2}^+} \vert \wt f\vert   
\big) \vert  x^\prime-  y\vert^\alpha .
\end{align*}

\medskip 

\noindent {\bf Case $x,y\in  B_{\delta R/256}^+$.}

We can assume that $x_n \leq y_n$.

We are going 
to consider separately the cases $\vert  x -  y\vert < y_n/4$ and 
$\vert  x -  y\vert \geq y_n/4$. 

\medskip 

\noindent {\bf Assume first that   $\vert  x -  y\vert < y_n/4$. }

Thus we have 
\begin{equation*}
 B_{\vert  x -  y\vert}(y)  \subset B_{y_n/2}(y)
 \subset B_{3y_n/5}(y)
 \subset B_{R/2}^+.
\end{equation*}
Recall that $v$ satisfies the linear elliptic equation (\ref{Eq.elliptic-2}).

We are going to apply the extension of the Krylov-Safonov H\"older estimate 
done by Gilbarg and Trudinger, 
Corollary 9.24 of \cite{G-T}, to the function $v$ with 
$\Omega, B_{R_0}$ and $B_R$ (of Corollary 9.24) substituted respectively by 
$B_{3y_n/5}(y)$, $B_{y_n/2}(y)$ and $B_{\vert x -y\vert}(y)$ (that is 
$R_0=y_n/2$ and $R=\vert x -y\vert$).

Following the notations of \cite[Section 9.7]{G-T}, the principal part of $\wt 
L_0$ is given 
by 
the symmetric matrix $\big(z_n \wt \alpha^{ij}(z)\big)$,
$z\in \Omega=B_{3y_n/5}(y)$. The functions $b_i, i=1,\dots,n$, are given by 
$b_i=2\wt \alpha^{in}$ (therefore $\vert b_i\vert \leq 
8\Lambda_{\wh K}$, see Equation (\ref{Eq.Linear})),
and $c=0$. For any $z\in \Omega=B_{3y_n/5}(y)$ we 
set  $\lambda_0(z) = z_n \dfrac{\lambda_{\wh K}}{2}$ and  
$\Lambda_0(z) = 4z_n \Lambda_{\wh K}$. 
Thus, 
for 
 any eigenvalue 
 $\rho_0 (z)$ of the symmetric matrix $\big(z_n \wt \alpha^{ij}(z)\big)$, we 
have 
 $\lambda_0(z) \leq \rho_0 (z) \leq \Lambda_0 (z)$, see the discussion after 
Equation (\ref{Eq.Linear}).  
  Therefore we can choose 
 $\gamma =  8\frac{\Lambda_{\wh K}}{\lambda_{\wh K}}$ 
to achieve 
\begin{equation*}
  \frac{\Lambda_0}{\lambda_0} \leq \gamma.
\end{equation*}

 For any $z\in \Omega=B_{3y_n/5}(y)$ we have 
\begin{equation*}
 \left( \frac{\vert b (z)\vert}{\lambda_0(z)}  \right)^2 \leq 
 n \left( 16 \frac{\Lambda_{\wh K}}{z_n \lambda_{\wh K}} \right)^2 \leq 
  n  \left( \frac{\Lambda_{\wh K}} {\lambda_{\wh K}} \right)^2  \, 
    \frac{40^2}{y_n^2}.
\end{equation*}
 Therefore we can choose 
 $\nu :=  n  \left( \frac{\Lambda_{\wh K}} {\lambda_{\wh K}} 
\right)^2  \, 
    \frac{40^2}{y_n^2}$
to achieve 
\begin{equation*}
  \left( \frac{\vert b \vert}{\lambda_0}  \right)^2 \leq \nu
\end{equation*}
on $B_{3y_n/5}(y)$.

\

Thus, in Corollary 9.24 of \cite{G-T} we have 
$\nu R_0^2= 400n \left( \frac{\Lambda_{\wh K}} {\lambda_{\wh K}} 
\right)^2 $, since $R_0= y_n/2$. 
Therefore, using  Corollary 9.24 of \cite{G-T} 
we obtain 
\begin{equation}\label{Eq.Local}
 \vert v(x) - v(y) \vert \leq C_1^\prime
 \frac{\vert  x -  y\vert^\eta}{y_n^\eta}
 \left( \text{osc}_{B_{y_n/2}(y)} v + 
 \omega_n y_n^2 \sup_{B_{y_n/2}(y)} 
\vert \wt f\vert \right)
\end{equation}
where $\omega_n$ is a constant depending only on $n$, and 
$C_1' >0$ and $\eta \in (0,1)$ are constant real numbers depending only 
on 
$n$, $\gamma$ and 
$\nu (y_n/2)^2 = 400n\left(\frac{\Lambda_{\wh K}}{\lambda_{\wh K}}\right)^2$, 
that is 
depending only on $n$ and $\frac{\Lambda_{\wh K}}{\lambda_{\wh K}}$. 

On the other hand we have
\begin{align*}
 \text{osc}_{B_{y_n/2}(y)} v &= \text{osc}_{B_{y_n/2}(y)} \big(v-v(y^\prime 
)\big) 
\\
&\leq 2 \sup_{z\in  B_{y_n/2}(y)} \vert v(z)-v(y^\prime )\vert \\
&\leq 2 \sup_{z\in  B_{y_n/2}(y)} \vert v(z)-v(z^\prime )\vert 
 + 2 \sup_{z\in  B_{y_n/2}(y)} \vert v(z^\prime )-v(y^\prime )\vert ,
\end{align*}
where $z=(z^\prime, z_n)$.

Observe that since $y\in  B_{\delta R/256}^+$ we have 
\begin{equation*}
 \Sigma_{y_n/2}(y^\prime ) \subset \Sigma_{\delta R/64}.
\end{equation*}
Consequently, for any $z\in B_{y_n/2}(y)$ we deduce from 
Proposition \ref{P.Holder.boundary}
\begin{align*}
 \vert v(z^\prime )- v(y^\prime )\vert &\leq 
  \frac{ C_1}{R^\alpha} 
 \big( \sup_{B_{R/2}^+} \vert Dw\vert + R \sup_{B_{R/2}^+} \vert \wt f\vert   
\big) \vert  z^\prime-  y^\prime\vert^\alpha \\
&\leq  \frac{ C_1}{R^\alpha} \, y_n^\alpha
 \big( \sup_{B_{R/2}^+} \vert Dw\vert + R \sup_{B_{R/2}^+} \vert \wt f\vert   
\big) . \\
\end{align*}
Note that for any $z\in B_{y_n/2}(y)$ we have 
\begin{equation*}
 B_{ z_n}^+(z^\prime ) \subset  B_{\frac{\delta}{4}\frac{R}{4}}^+(z^\prime ) 
\quad 
\text{and} \quad  B_{\frac{R}{4}}^+(z^\prime )  \subset B_{R/2}^+. 
\end{equation*}
 Moreover, by Proposition \ref{P.Holder.boundary} $v$ extends 
continuously to $z^\prime $. We deduce from the inequality (\ref{F.Krylov}) 
that 
for  any $z\in B_{y_n/2}(y)$ we have
\begin{align*}
 \left\vert v(z) - v(z^\prime)  \right\vert &\leq 
 \frac{C_1}{  R^\alpha}\, z_n^\alpha \big( 
\sup_{B_{R/2}^+} \vert Dw\vert + 
{  R}\sup_{B_{R/2}^+} \vert \wt f\vert   \big)  \\
&\leq \left(\frac{3}{2}\right)^\alpha \frac{C_1}{  R^\alpha}\, y_n^\alpha \big( 
\sup_{B_{R/2}^+} \vert Dw\vert + 
{  R}\sup_{B_{R/2}^+} \vert \wt f\vert   \big).
\end{align*}
Therefore we obtain
\begin{equation}\label{Eq.Local-2}
  \text{osc}_{B_{y_n/2}(y)} v \leq 5
  \frac{C_1}{  R^\alpha}\, y_n^\alpha \big( 
\sup_{B_{R/2}^+} \vert Dw\vert + 
{  R}\sup_{B_{R/2}^+} \vert \wt f\vert   \big).
\end{equation}
Now we set $\beta=\min (\alpha,\eta)$. Since 
$\vert x-y\vert \leq y_n/4$ we have 
\begin{equation*}
 \frac{\vert  x -  y\vert^\eta}{y_n^\eta} \leq 
 \frac{\vert  x -  y\vert^\beta}{y_n^\beta} \qquad \text{and} \qquad 
 \frac{y_n^\alpha}{R^\alpha} \leq  \frac{y_n^\beta}{R^\beta} .
\end{equation*}
Therefore we deduce from (\ref{Eq.Local})  and (\ref{Eq.Local-2}) 
that 
\begin{align*}
 \vert v(x) - v(y) \vert &\leq C_1^\prime 
 \frac{\vert  x -  y\vert^\beta}{y_n^\beta}
  \left( 5
  \frac{C_1}{  R^\beta}\, y_n^\beta \big( 
\sup_{B_{R/2}^+} \vert Dw\vert + 
{  R}\sup_{B_{R/2}^+} \vert \wt f\vert   \big)
   +  
\omega_n y_n^2 \sup_{B_{y_n/2}(y)} 
\vert \wt f\vert \right) \\
 &\leq C_1^\prime 
 \vert  x -  y\vert^\beta
  \left( 5
  \frac{C_1}{  R^\beta}\big( 
\sup_{B_{R/2}^+} \vert Dw\vert + 
{  R}\sup_{B_{R/2}^+} \vert \wt f\vert   \big)
   +  
\omega_n y_n^{2-\beta} \sup_{B_{y_n/2}(y)} 
\vert \wt f\vert \right) \\
&\leq C_1^\prime 
 \vert  x -  y\vert^\beta
  \left( 5
  \frac{C_1}{  R^\beta}\big( 
\sup_{B_{R/2}^+} \vert Dw\vert + 
{  R}\sup_{B_{R/2}^+} \vert \wt f\vert   \big)
   +  
\omega_n R^{2-\beta} \sup_{B_{y_n/2}(y)} 
\vert \wt f\vert \right). \\
\end{align*}
Consequently, setting $C_1^{\prime\prime}= C_1^\prime\big(5C_1  + \omega_n 
R\big)$, we obtain
\begin{equation*}\label{Eq.Local-3}
  \vert v(x) - v(y) \vert \leq \frac{C_1^{\prime\prime}}{R^\beta}  
 \Big( 
\sup_{B_{R/2}^+} \vert Dw\vert + 
{  R}\sup_{B_{R/2}^+} \vert \wt f\vert   \Big)
 \vert  x -  y\vert^\beta 
\end{equation*}
where $C_1^{\prime\prime} = C_1^{\prime\prime}(n, K, 
\lvert \varphi\lvert_{C^1(\Omega)}, \lambda_{\wh K},\Lambda_{\wh K},R)>0$  
 and 
$\beta=\beta  (n, K, 
\lvert \varphi\lvert_{C^1(\Omega)}, \lambda_{\wh K},\Lambda_{\wh K}) 
\in (0,1)$, $\beta \leq \alpha$.

\medskip

\noindent {\bf  Assume now that $\vert  x -  y\vert \geq y_n /4$.}

Recall that we are also assuming that $x_n \leq y_n$. 

Observe that  
\begin{align*}
 B_{x_n }^+(x^\prime )  \subset B_{\frac{\delta}{4}\frac{R}{8}}^+(x^\prime ) 
 \subset B_{\delta R/16}^+
\qquad \text{and}\qquad  B_{R/8}^+(x^\prime ) \subset B_{R/2}^+, \\
 B_{y_n }^+(y^\prime )  \subset B_{\frac{\delta}{4}\frac{R}{8}}^+(y^\prime ) 
 \subset B_{\delta R/16}^+
\qquad \text{and}\qquad  B_{R/8}^+(y^\prime ) \subset B_{R/2}^+.
\end{align*}
Since $v$ can be extended to a continuous function on 
$B_{\delta R/16}^+ \cup \Sigma_{\delta R/16}$ we get from the inequality 
(\ref{F.Krylov})  applied successively on the half balls 
$ B_{x_n }^+(x^\prime )$ and $ B_{y_n }^+(y^\prime )$

\begin{align*}
 \left\vert v(x) - v(x^\prime)  \right\vert  & \leq 
 \frac{C_1}{  R^\alpha}x_n^\alpha \big( 
\sup_{B_{R/2}^+} \vert Dw\vert + 
{  R}\sup_{B_{R/2}^+} \vert \wt f\vert   \big) \\
 \left\vert v(y) - v(y^\prime)  \right\vert  & \leq 
 \frac{C_1}{  R^\alpha}y_n^\alpha \big( 
\sup_{B_{R/2}^+} \vert Dw\vert + 
{  R}\sup_{B_{R/2}^+} \vert \wt f\vert   \big). \\
\end{align*}
Moreover, since $x^\prime,y^\prime\in \Sigma_{\delta R/64}$,  
Proposition \ref{P.Holder.boundary} gives
\begin{equation*}
 \vert v( x^\prime)- v( y^\prime)\vert \leq \frac{ C_1}{R^\alpha} 
 \big( \sup_{B_{R/2}^+} \vert Dw\vert + R \sup_{B_{R/2}^+} \vert \wt f\vert   
\big) \vert  x^\prime-  y^\prime\vert^\alpha .
\end{equation*}
Therefore, using the above inequalities, $x_n \leq y_n$, 
$y_n \leq 4\vert x-y\vert$ and 
$\vert x^\prime-y^\prime\vert \leq \vert x-y\vert$, 
 we get 
\begin{align*}
  \vert v( x)- v( y)\vert &\leq \vert v(x)- v(x^\prime ) \vert  + 
  \vert v(x^\prime )- v(y^\prime )\vert + \vert v(y^\prime )- v(y) \vert \\
  &\leq \frac{ C_1}{R^\alpha} \big( 2 y_n^\alpha + 
  \vert x^\prime - y^\prime \vert^\alpha  \big) 
  \big( \sup_{B_{R/2}^+} \vert Dw\vert + R \sup_{B_{R/2}^+} \vert \wt f\vert   
\big)\\ 
&\leq 9 \frac{ C_1}{R^\alpha}
\big( \sup_{B_{R/2}^+} \vert Dw\vert + R \sup_{B_{R/2}^+} \vert \wt f\vert   
\big) \vert x-y\vert^\alpha \\
&\leq 9 \frac{ C_1}{R^\beta}
\big( \sup_{B_{R/2}^+} \vert Dw\vert + R \sup_{B_{R/2}^+} \vert \wt f\vert   
\big) \vert x-y\vert^\beta,
\end{align*}
since $\frac{\vert x-y\vert^\alpha}{R^\alpha}\leq 
\frac{\vert x-y\vert^\beta}{R^\beta}$.

\bigskip 

Setting $C_2:=\max \big(9C_1, C_1^{\prime\prime}\big)$ and 
considering each 
case, 
 we conclude that for any 
$x,y \in  B_{\delta R/256}^+ \cup \Sigma_{\delta R/256}$ we have 
\begin{equation*}
 \vert v( x)- v( y)\vert \leq \frac{C_2}{R^\beta} 
 \big( \sup_{B_{R/2}^+} \vert Dw\vert + R \sup_{B_{R/2}^+} \vert \wt f\vert   
\big) \vert  x-  y\vert^\beta .
\end{equation*}
By construction, $C_2$  depends on 
$n,K,\lvert \varphi\lvert_{C^1(\Omega)}, 
\lambda_{\wh K},\Lambda_{\wh K}$ and $R$, and  $\beta$ depends on 
$n,K$, $\lvert \varphi\lvert_{C^1(\Omega)}, 
\lambda_{\wh K}$ and $\Lambda_{\wh K}$, 
as desired.
\end{proof}

\end{document}